\documentclass[12pt, ams fonts]{amsart}

\usepackage{color}

\usepackage{amssymb,amsmath}

\usepackage{graphicx}

\usepackage[vcentermath,enableskew]{youngtab}

\input xy
\xyoption{all}

\font\teneufm=eufm10 \font\seveneufm=eufm7 \font\fiveeufm=eufm5
\newfam\eufmfam
\textfont\eufmfam=\teneufm \scriptfont\eufmfam=\seveneufm
\scriptscriptfont\eufmfam=\fiveeufm

\newcommand{\C}{\mathbb{C}}

\newcommand{\Z}{\mathbb{Z}}

\newcommand{\A}{\mathcal A}

\newcommand{\g}{\mathfrak{g}}
\newcommand{\h}{\mathfrak{h}}
\newcommand{\vareps}{\varepsilon}

\DeclareMathOperator{\Supp}{Supp}
\DeclareMathOperator{\Span}{Span} 
\DeclareMathOperator{\Hom}{Hom} \DeclareMathOperator{\ch}{ch}

 \DeclareMathOperator{\ad}{ad}
 
\DeclareMathOperator{\Id}{Id}

\DeclareMathOperator{\wht}{wht}

\DeclareMathOperator{\Der}{Der}

\numberwithin{equation}{section}

\newtheorem{definition}{Definition}[section]

\newtheorem{example}[definition]{Example}

\theoremstyle{remark}

\newtheorem{remark}[definition]{Remark} 

\theoremstyle{plain} 

\newtheorem{theorem}[definition]{Theorem}
\newtheorem{lemma}[definition]{Lemma}

\newtheorem{proposition}[definition]{Proposition}

\def\Z{\mathbb Z}
\def\C{\mathbb C}

\begin{document}

\title{Bounded weight modules of the Lie algebra of vector fields on ${\mathbb C}^2$}

\author{Andrew Cavaness and Dimitar Grantcharov$^1$}

\address{Department of Mathematics \\
         University of Texas at Arlington \\ Arlington, TX 76021, USA}

         \email{cavaness@uta.edu}
         
\address{Department of Mathematics \\
         University of Texas at Arlington \\ Arlington, TX 76021, USA}

         \email{grandim@uta.edu}

\thanks{$^{1}$This work was partially supported by Simons Collaboration Grant 358245}

\maketitle
\begin{abstract}
We study weight modules of the Lie algebra $W_2$  of vector fields on $\C^2$. A classification of all simple weight modules of $W_2$ with a uniformly bounded set of weight multiplicities is provided. To achieve this classification we introduce a new family of generalized tensor $W_n$-modules. Our classification result is an important step in the  classification of all simple weight  $W_n$-modules with finite weight multiplicities.\\

\medskip\noindent 2000 MSC: 17B66, 17B10 \\

\noindent Keywords and phrases: Lie algebra, Cartan type, weight module, localization.

\end{abstract}

\section*{Introduction}
Lie algebras of vector fields have been studied since the appearance of infinite Lie groups in the works of S. Lie in the late 19th century. Based on the fundamental works of E. Cartan in the early 20th century, some of these infinite-dimensional Lie algebras are known as Lie algebras of Cartan type. A classical example of a Cartan type Lie algebra is  the Lie algebra $W_n$ consisting of derivations of the polynomial algebra $\C[x_1,...,x_n]$, or, equivalently, the Lie algebra of polynomial vector fields on $\C^n$. The first classification results concerning representations of $W_n$ and other Cartan type Lie algebras were  obtained by A. Rudakov in 1974-1975, \cite{Rud1}, \cite{Rud2}.  These results address the classification of a class of irreducible $W_n$-representations that satisfy some natural topological conditions. The so-called {\it tensor modules}, that is modules $T({\boldsymbol{\nu}}, S)$ whose underlying spaces are tensor products  $\boldsymbol{x}^{\boldsymbol{\nu}}\C[x_1^{\pm 1},..,x_n^{\pm 1}] \otimes S$ of a ``shifted'' Laurent polynomial ring and a finite-dimensional $\mathfrak{gl}_n$-module $S$, play an important role in the works of Rudakov. Tensor  $W_1$-modules and extensions  of tensor modules were studied extensively in the 1970's and in the 1980's by B. Feigin, D. Fuks, I. Gelfand, and others, see for example, \cite{FF}, \cite{Fuk}.

In this paper we focus on the category of weight representations of $W_n$, namely those that decompose as direct sums of weight spaces relative to the subalgebra $\mathfrak{h}_{W_n}$  of $W_n$ spanned by the derivations $x_1\partial_1$,...,$x_n\partial_n$. Weight representations of Lie algebras of vector fields (in particular, of $W_n)$) are subject of interest by both mathematicians and theoretical physicists in the last 30 years. Another important example of a Lie algebra of vector fields is the Witt algebra $\rm{Witt}_n$ consisting of the derivations of the Laurent polynomial algebra $\C[x_1^{\pm},...,x_n^{\pm 1}]$, or, equivalently, the Lie algebra of polynomial vector fields on the $n$-dimensional complex torus. In particular, $\rm{Witt}_1$ is the centerless Virasoro algebra. The classification of all simple weight representations with finite weight multiplicities of  $W_1$ and ${\rm Witt}_1$  (and hence of the Virasoro algebra) was obtained by O. Mathieu in 1992, \cite{M-Vir}, proving a conjecture of V. Kac, \cite{Kac}. Following a sequence of works of S. Berman, Y. Billig,  C. Conley, S. Eswara Rao, X. Guo, C. Martin, O. Mathieu, V. Mazorchuk, V. Kac, G. Liu, R. Lu, A. Piard, Y. Su, K. Zhao, very recently, Y. Billig and V. Futorny managed to extend Mathieu's classification result to ${\rm Witt}_n$ for arbitrary $n\geq 1$ (see \cite{BF} and the references therein). The classification theorem in  \cite{BF} states roughly that every nontrivial simple weight ${\rm Witt}_n$-module with finite weight multiplicities is either a submodule of a tensor module or a module of highest weight type. 

In contrast with ${\rm Witt}_n$, the classification of the simple weight $W_n$-modules $M$ with finite weight multiplicities is still an open problem for $n>1$. The possible supports (sets of weights)  of all such $M$ have been described by I. Penkov and V. Serganova in \cite{PS}. In addition, in  \cite{PS}, a parabolic induction theorem for such modules $M$ is proven. More precisely, it is shown that $M$ is a quotient of a parabolically induced module from a parabolic subalgebra $\mathfrak p$ of $W_n$. Unfortunately, the parabolic subalgebras $\mathfrak p = \mathfrak l \oplus \mathfrak n^+$ that appear in the parabolic induction theorem are quite complicated and have Levi components $\mathfrak l$ isomorphic to a semi-direct sum of Lie algebras of Cartan type and finite dimensional-reductive Lie algebras. Another obstacle in the study of weight $W_n$-modules is the fact that the $W_n$-modules $T({\boldsymbol{\nu}}, S)$ are highly reducible - they  may contain $2^n$ simple subquotients. 

The purpose of this paper is to make the first step towards the classification of the simple weight $W_n$-modules with finite weight multiplicities. Namely, we classify the simple bounded  modules of $W_2$, that is, all simple weight $W_2$-modules whose sets of weight multiplicities is uniformly bounded.  The classification is given in Theorem \ref{th-main} (the tensor modules are introduced in Definition \ref{def-tensor}). The second step in the weight module classification is to classify all simple bounded  $\mathfrak l$-modules, where  $\mathfrak l$ is a Levi subalgebra of a parabolic subalgebra of $W_n$. The last step is, based on the parabolic induction theorem of Penkov-Serganova, to complete the classification in question. The second and the third steps will be addressed in a subsequent paper. We note that for $n=2$, the classification of simple bounded $\mathfrak l$-modules, where $\mathfrak l$ is a Levi subalgebra of a parabolic subalgebra of $W_2$, is obtained in the present paper and, in fact, is used to classify the simple bounded $W_2$-modules. It turns out that in this case $\mathfrak l \simeq \Der \C[x] \ltimes \C[x]$. 

In addition to obtaining the classification of simple weight modules with finite weight multiplicities, the results in the present paper will be essential for the study  of the category  $\mathcal B$ of bounded representations of $W_n$. This category is intimately related to  the corresponding category of bounded $\mathfrak{sl}_{n+1}$-modules. It is expected that, like in the case of $\mathfrak{sl}_{n+1}$, the indecomposable injectives of $\mathcal B$  will have a nice geometric realizations in terms of twisted functions and twisted differential forms on algebraic varieties,  \cite{GS2},  \cite{GS3}.

An important tool  in the present paper is the twisted localization functor, a functor used by O. Mathieu in the proof of another fundamental result: the classification of all simple weight modules with finite weight multiplicities of finite-dimensional reductive Lie algebras, \cite{M}. Also, in order to deal with the reducibility of $T({\boldsymbol{\nu}}, S)$, we introduce a family of (generalized) tensor modules  $T({\boldsymbol{\nu}}, S, J)$. The modules $T({\boldsymbol{\nu}}, \boldsymbol{\lambda}, J) = T({\boldsymbol{\nu}}, S, J)$ are defined for  a tuple $J = (a_1,...,a_k)$  of signed integers $a_i = b_i^+$ or $a_i = b_i^-$, where the $b_i$'s are in the set of all indices $j$ such that $\lambda_j - \nu_j \in \Z$, and $\boldsymbol{\lambda}$ is the highest weight of $S$. Our main result is that all nonzero simple bounded  $W_2$-modules are isomorphic to $T({\boldsymbol{\nu}}, \boldsymbol{\lambda}, J) $ for some ${\boldsymbol{\nu}}, \boldsymbol{\lambda}, J$. A similar result will hold for the simple bounded weight $W_n$-modules ($n\geq 2$) as well, but   additional conditions for the tensor modules $T({\boldsymbol{\nu}}, \boldsymbol{\lambda}, J)$ corresponding to fundamental weights $\boldsymbol{\lambda}$ have to be imposed.

The content of the paper is as follows. In Section 2 we collect important results on the twisted localization functor, parabolic subalgebras and tensor modules of $W_n$. In particular we  provide an explicit list of  the possible parabolic subalgebras $\mathfrak p$ of $W_2$. In Section 3, we classify all simple bounded modules over the Lie algebra $\A = \Der \C[x] \ltimes \C[x]$. In Section 4, based on the results of Section 3, we complete the classification of simple bounded $W_2$-modules.

\medskip
\noindent{\it Acknowledgements.} We would like to thank M. Gorelik  and V. Serganova  for the fruitful discussions and helpful suggestions. We also would like to thank the referee for the valuable remarks and the careful reading of the manuscript. 

\section{Notation and Conventions}

Throughout the paper the ground field is $\mathbb C$. All vector spaces, algebras, and tensor products are assumed to be over $\mathbb C$ unless otherwise stated. 

By $W_n$ we denote the Lie algebra $ \Der \C [x_1,...,x_n] $ of derivations of ${\mathbb C} [x_1,...,x_n]$. Also, $\A_n$ will be  the semi-direct product $\A_n= \Der  \C [x_1,...,x_n] \ltimes \C [x_1,..,x_n]$. For simplicity we set $\A := \A_1$ and  $\partial_i :=\frac{\partial}{\partial x_i}$. Every element $w$ of $W_n$ can be written uniquely as $w= \sum_{i=1}^n f_i\partial_i$, for some $f_i \in \C[x_1,...,x_n]$.

By $ {\mathbb Z}_{\geq k}$ we denote the set of all integers $n$ such that $n \geq k$. We similarly define  $ {\mathbb Z}_{\leq k}$,  $ {\mathbb Z}_{> k}$,  $ {\mathbb R}_{\geq k}$, etc. If $M$ is a set of real numbers, and $S$ is
a subset of a real vector space $V$, then by $MS$ we denote the set of all $M$-linear combinations of elements in $S$. 

For a Lie algebra $\mathfrak a$ by $U(\mathfrak a)$ we denote the universal enveloping algebra of $\mathfrak a$.

Throughout the paper we use the multi index-notation for monomials: $\boldsymbol{x}^{\boldsymbol{\nu}} = x_1^{\nu_1}...x_n^{\nu_n}$ if 
$\boldsymbol{x} = (x_1,...,x_n)$ and $\boldsymbol{\nu} = (\nu_1,...,\nu_n)$. If $n$ is fixed, we set ${\mathbb C}[\boldsymbol{x}] = {\mathbb C}[x_1,...,x_n]$, ${\mathbb C}[\boldsymbol{x}^{\pm 1}] = {\mathbb C}[x_1^{\pm 1},...,x_n^{\pm 1}]$, and $\boldsymbol{x}^{\boldsymbol{\nu}}  {\mathbb C}[\boldsymbol{x}^{\pm 1}] = x_1^{\nu_1}...x_n^{\nu_n}{\mathbb C}[x_1^{\pm 1},...,x_n^{\pm 1}]$, where the latter is the span of all (formal) monomials $x_1^{\nu_1+k_1}...x_n^{\nu_n + k_n}$, $k_i \in {\mathbb Z}$. 

For an $n$-tuple $\boldsymbol{\nu} = (\nu_1,...,\nu_n)$ in ${\mathbb C}^n$, we set $\mbox{Int} (\boldsymbol{\nu}):=\{ i \; | \;  \nu_i \in {\mathbb Z}\}$.

\section{Preliminaries}

\subsection{The Lie algebra $\A$}

Recall that   $\A = \Der \C[x] \ltimes \C[x]$ and $W_1 =  \Der \C[x]$. Note that by definition $[D,f] = Df$ for $D \in \Der \C[x]$ and $f \in \C[x]$. In terms of generators and relations, $\mathcal A$ can be defined as follows: $$\A = \Span \{D_i, I_j \; | \; i \in \Z_{\geq -1},  j \in \Z_{\geq 0} \}$$ with
\begin{eqnarray*}
\left[D_i, D_j\right] &=& (j-i) D_{i+j},\\
\left[D_i, I_j\right] &=& j I_{i+j},\\
\left[I_i, I_j\right] &=& 0.
\end{eqnarray*}
Here $D_i$ and $I_j$ correspond to $x^{i+1} \partial$ and $x^j$, respectively. Note that the center of $\A$ is  generated by $I_0$. We say that an ${\mathcal A}$-module $M$ has central charge $c$ if $I_0m = cm$ for every $m \in M$. In particular, every irreducible ${\mathcal A}$-module $M$ has a central charge.

We say that $\h \subset \A$ is a \emph{Cartan subalgebra} of $\A$ if $\h$ is both self-normalizing and nilpotent. In what follows, we fix the Cartan subalgebra of $\mathcal A$ to be $\h_{\mathcal A} = \Span \{D_0, I_0\}$. We also have the triangular decomposition $\A = \A^- \oplus \A^0 \oplus \A^+$, where $\A^- = \Span \{ D_{-1}\}$, $\A^0 = \h_{\A}$, and $\A^+ = \Span \{D_i, I_j \; | \; i,j\geq 1\}$. Define $\varepsilon, \delta \in \h_{\A}^*$ by the identities $$\varepsilon(D_0) = 1, \; \; \varepsilon(I_0) = 0; \; \; \delta(D_0) = 0, \; \; \delta(I_0)=1.$$ 

\subsection{Injective and finite actions}
Let $\g$ be Lie algebra, and $M$ be a $\g$-module. We say that an element $x$ of $\g$ acts \emph{locally nilpotently} (or, \emph{finitely}) on a vector $m$ in $M$, if there is $N = N(x,m)$ such that $x^N(m) = 0$. If such $N$ does not exists we say that $x$ {\it acts injectively} on $m$. We say that $x$ acts injectively (respectively, finitely) on $M$ if $x$ acts injectively  (respectively, finitely) on all $m \in M$. 

We will often use the following setting. Let $x$ be an $\ad$-nilpotent element in $\g$ and let $M$ be a $\g$-module. Then the set $M^{\langle x \rangle}$ of all $m$ on which $x$ acts finitely is a submodule of $M$. In particular, if $M$ is simple, then every ad-nilpotent element $x$ of $\g$ acts either finitely or injectively on $M$.

\subsection{Weight modules} \label{subsec-wht} We first introduce weight modules in a general setting.
Let $\mathcal U$ be an associative unital algebra and $\mathcal H\subset\mathcal U$
be a commutative subalgebra. We assume in addition that 
$\mathcal H$  
is a polynomial algebra identified with the symmetric algebra of a vector space ${\mathfrak h}$, and that we
have a decomposition
$$\mathcal U=\bigoplus_{\mu\in {{\mathfrak h}^*}}\mathcal U^\mu,$$
where
$$\mathcal U^\mu=\{x\in\mathcal U | [h,x]=\mu(h)x, \forall h\in\mathfrak h\}.$$
Let $Q_{\mathcal U} = {\mathbb Z}\Delta_{\mathcal U}$ be the ${\mathbb Z}$-lattice in  ${\mathfrak h}^*$
generated by $\Delta_{\mathcal U}= \{ \mu \in {\mathfrak h}^* \; | \; {\mathcal U}^{\mu} \neq 0\}$. We  obviously have
${\mathcal U}^\mu {\mathcal U}^\nu\subset {\mathcal U}^{\mu+\nu}$.

We call a ${\mathcal U}$-module $M$ {\it a generalized weight $({\mathcal U}, {\mathcal H})$-module} (or just\emph{ generalized weight $\mathcal U$-module}) if  $M = \bigoplus_{\lambda \in {\mathfrak h}^*} M^{(\lambda)}$, where  
$$
M^{(\lambda)} = \{m\in M \; |\;  (h- \lambda (h)\mbox{Id})^N m=0\,\text{for some}\, N>0\, \text{and all}\, h \in \mathfrak h\}.
$$
We call $M^{(\lambda)}$ the generalized weight space of $M$ and $\dim
M^{(\lambda)}$ the weight multiplicity of the weight $\lambda$. A vector $v$ in $M^{(\lambda)}$ is called a \emph{weight vector of weight $\lambda$} and we write $\wht(v) = \lambda$.
Note that 
\begin{equation}\label{rootweight}
\mathcal U^\mu M^{(\lambda)}\subset M^{(\mu+\lambda)}.
\end{equation}
A generalized weight module $M$ is called a {\it weight $({\mathcal U}, {\mathcal H})$-module} (or just \emph{weight} \emph{$\mathcal U$-module}) if $M^{(\lambda)}  = M^{\lambda},$ where 
$$
M^{\lambda} = \{m\in M \; |\;  (h - \lambda(h) \mbox{Id}) m=0\,\text{ for all } h \in {\mathfrak h}\}.
$$

In the case when ${\mathcal U} = U(\mathfrak g)$ is the universal enveloping algebra of a Lie algebra $\mathfrak g$ and $\h$ is a subalgebra of $\mathfrak g$, a (generalized) weight $({\mathcal U}, {\mathcal H})$-module will be called (generalized) weight $(\mathfrak g, \mathfrak h)$-module.

\begin{definition}
\begin{itemize}
\item[(i)]An $\A$-module $M$  is a \emph{weight $\A$-module} if $M$ is a weight $(\A , \mathcal{H})$-module for $\mathcal H = \C [\h_{\mathcal A}]$.  If $M$ is a weight $\A$-module we call the set of weights $\lambda \in \h_{\A}^*$ such that $M^{\lambda} \neq 0$ the \emph{$\A$-support} (or simply the \emph{support}) of $M$ and denote it by $\Supp M$.

\item[(ii)]We say that a weight $\A$-module $M$ is \emph{bounded} if there is $N>0$ such that $\dim M^\lambda < N$ for all $\lambda \in \h_{\A}^*$. If $M$ is bounded we call $\sup \{ \dim M^{\lambda}\; | \; \lambda \in \h_{\A}^* \}$ the $\A$-\emph{degree} (or simply the \emph{degree}) of $M$. 
\end{itemize}
\end{definition}

The adjoint module $\A$ is a weight module of central charge $0$ such that $\A^{\lambda} \neq 0$ if and only if $\lambda = n \varepsilon$ for $n \in \Z_{\geq -1}$. The set $\Delta_{\A} = \{ - \varepsilon, n\varepsilon \; | \; n\in \Z_{>0}\}$ is the \emph{root system} of $\A$, and
$$
\A^{-\varepsilon} = \Span \{ D_{-1}\}; \, \A^{n\varepsilon} = \Span \{ I_n, D_n \},  n\in \Z_{>0}
$$
are the root spaces of $\A$. 

If $M$ has central charge $c$, then a weight of $M$ is of the form $\lambda \varepsilon +  c \delta$, for some $\lambda \in \C$. If $c$ is fixed, with a slight abuse of notation we set $M^{\lambda} = M^{\lambda \varepsilon +  c \delta}$, for all weight modules $M$ with central charge $c$. In particular,  $M = \bigoplus _{\lambda \in \C} M^\lambda$ and $\Supp M \subset \C$.

We similarly introduce the notions of weight and bounded $W_n$-modules. More precisely, let $\h_{W_n}$ (or simply  $\h_{W}$ if $n$ is fixed) be the subalgebra $\Span \{ x_1\partial_1,..., x_n\partial_n \}$. Then $\h_W$ is a Cartan subalgebra of $W_n$.

\begin{definition}
A $W_n$-module $M$  is a \emph{weight $W_n$-module} if $M$ is a weight $(W_n, {\mathcal H})$-module with $\mathcal H = \C [\h_W]$. We say that a weight $W_n$-module $M$ is \emph{bounded} if there is $N>0$ such that $\dim M^\lambda < N$ for all $\lambda \in \h_{W}^*$. If $M$ is bounded we call $\sup \{ \dim M^{\lambda}\; | \; \lambda \in \h_{W}^* \}$ the $W_n$-\emph{degree} (or simply the \emph{degree}) of $M$. 
\end{definition}

Note that $W_n$ is a weight $W_n$-module whose support is $\Delta_{W_n} \cup \{ 0\}$, where
$\Delta_{W_n}$ is the root system of $W_n$.  We identify the root lattice of $W_n$ with ${\mathbb Z}^n$ and will often write every element $\alpha$ of ${\mathbb Z} \Delta_{W_n}$ as an $n$-tuple $(\alpha_1,...,\alpha_n)$ of integers. In particular,
$$
\Delta_{W_n} \cup \{ 0\}= \{ (\alpha_1,...,\alpha_n) \; | \; \alpha_i \geq 0\} \sqcup  \{ (\alpha_1,...,\alpha_n) \; | \; \exists i: \alpha_i  = -1 \mbox{ and } \alpha_j \geq 0 \mbox{ for all } j\neq i \}.
$$
For simplicity we will often write $\Delta_{W}$ for $\Delta_{W_n}$. In what follows we use the (root) basis of $W_n$ consisting of the elements $\boldsymbol{x}^{\boldsymbol{\alpha}}  (x_i\partial_i)$, $\boldsymbol{\alpha} \in \Delta_{W_n} \cup \{ 0\}$, $i=1,...,n$. 

\subsection{Tensor modules} 
We say that $(\lambda_1,...,\lambda_n) \in {\mathbb C}^n$ is a dominant integral $\mathfrak{gl}_n$-weight if $\lambda_i - \lambda_{i+1} \in {\mathbb Z}_{\geq 0}$ for all $i = 1,..,n-1$. If $\boldsymbol{\lambda} = (\lambda_1,...,\lambda_n)$ is a dominant integral weight, by $L_{\mathfrak{gl}} (\boldsymbol{\lambda} ) = L_{\mathfrak{gl}} (\lambda_1,...,\lambda_n )$ we denote the simple finite-dimensional module with highest weight  $\boldsymbol{\lambda}$. 

For a dominant integral $\mathfrak{gl}_n$-weight $\boldsymbol{\lambda} = (\lambda_1,...,\lambda_n)$ and any $\boldsymbol{\nu} = (\nu_1,...,\nu_n)$ in $\C^n$, we define the $W_n$-modules $T(\boldsymbol{\nu},\boldsymbol{\lambda})$ as follows:
$$
T(\boldsymbol{\nu},\boldsymbol{\lambda}) = \boldsymbol{x}^{\boldsymbol{\nu}} {\mathbb C}[\boldsymbol{x}^{\pm 1}] \otimes L_{\mathfrak{gl}} (\boldsymbol{\lambda} ) 
$$
with $W_n$-action defined by
\begin{equation} \label{def-tensor-action}
( \boldsymbol{x}^{\boldsymbol{\alpha}} x_i\partial_i)\cdot (\boldsymbol{x}^{\bf s} \otimes v) = s_i  \boldsymbol{x}^{\boldsymbol{\alpha + s}} \otimes v+ \sum_{j=1}^n \alpha_j \boldsymbol{x}^{\boldsymbol{\alpha + s}} \otimes E_{ji}v,
\end{equation}
where $\boldsymbol{\alpha} \in \Delta_{W_n} \cup\{ 0 \}$, ${\bf s} \in \boldsymbol{\nu} + {\mathbb Z}^n$, $v \in  L_{\mathfrak{gl}} (\boldsymbol{\lambda} )$, and $E_{ji}$ is the $(j,i)$th elementary matrix of $\mathfrak{gl}_n$. As indicated in the introduction, these modules play important role in the classification of simple weight modules with finite weight multiplicities over various classes of Lie algebras. We  easily extend the $W_n$-action on $T(\boldsymbol{\nu},\boldsymbol{\lambda})$   to an ${\mathcal A}_n$-action. Namely, for $c \in {\mathbb C}$ we define $T(\boldsymbol{\nu},\boldsymbol{\lambda}, c)  = T(\boldsymbol{\nu},\boldsymbol{\lambda}) $ as vector space and set
\begin{equation} \label{x-i-action}
\boldsymbol{x}^{\bf j}\cdot (\boldsymbol{x}^{\bf s} \otimes v) := c \boldsymbol{x}^{\bf{j + s}} \otimes v.
\end{equation}
The next theorem gives a necessary and sufficient condition when two tensor modules are isomorphic as $W_n$-modules and ${\mathcal A}_n$-modules. The fact is well-known but for reader's convenience a short proof suggested by M. Gorelik is provided. 
\begin{proposition}
The following are equivalent.
\begin{itemize}
\item[(i)] $T(\boldsymbol{\nu},\boldsymbol{\lambda}) \simeq T(\boldsymbol{\nu}',\boldsymbol{\lambda}')$ as $W_n$-modules.
\item[(ii)] $T(\boldsymbol{\nu},\boldsymbol{\lambda}) \simeq T(\boldsymbol{\nu}',\boldsymbol{\lambda}')$ as ${\mathcal A}_n$-modules.
\item[(iii)] $\boldsymbol{\nu} - \boldsymbol{\nu}' \in {\mathbb Z}^n$ and $\boldsymbol{\lambda} = \boldsymbol{\lambda}'$.
\end{itemize}
\end{proposition}
\begin{proof}
The fact that (iii) implies (i) and (ii) is straightforward. Also, obviously (ii) implies (i). It remains to show that (i) implies (iii). 

Let $\psi: T(\boldsymbol{\nu},\boldsymbol{\lambda})\to T(\boldsymbol{\nu}' ,\boldsymbol{\lambda}' )$ be an isomorphism.  Since the $\boldsymbol{\mu}$-weight space of $T(\boldsymbol{\nu},\boldsymbol{\lambda})$ is $\boldsymbol{x}^{\boldsymbol{\mu}} \otimes L(\boldsymbol{\lambda})$, we have that for every $\bf{s} \in \boldsymbol{\nu} + \Z^n$ and $u \in L(\boldsymbol{\lambda})$, $\psi (\boldsymbol{x}^{\boldsymbol{s}} \otimes u) = \boldsymbol{x}^{\boldsymbol{s}} \otimes u'$ for some  $u' \in L(\boldsymbol{\lambda}')$. Also, $\dim L(\boldsymbol{\lambda}) = \dim L(\boldsymbol{\lambda}')$.

Let $v_{\lambda}$ be a highest weight vector of $L(\boldsymbol{\lambda})$, and let us fix ${\bf s} \in \boldsymbol{\nu} + \Z^n$ such that $s_i \neq 0$ for every $i$. Also, let $\psi(\boldsymbol{x}^{\bf s}\otimes v_{\lambda})=\boldsymbol{x}^{\bf s} \otimes v$ for some
$v\in L(\boldsymbol{\lambda}')$. Denote by $v_i$  the $\mathfrak{gl}_n$-weight components of $v$, i.e. 
$v=\sum_{i=1}^t v_i$, where  
$v_i\in L(\boldsymbol{\lambda})^{\boldsymbol{\eta}_i}$ are nonzero vectors and $\boldsymbol{\eta}_1,\ldots,\boldsymbol{\eta}_t$
are distinct weights (of $\mathfrak{gl}_n$).
Assume that $\boldsymbol{\eta}_1$ is a minimal element in $\{\boldsymbol{\eta}_1,..., \boldsymbol{\eta}_n\}$ with respect to the standard
partial order on ${\mathfrak h}_{\mathfrak{gl}_n}^*$. Then for $1\leq j<i\leq n$ we have 
\begin{eqnarray*}
(x_ix_j \partial_i)\partial_j(\boldsymbol{x}^{\boldsymbol{s}}\otimes v)&=&\boldsymbol{x}^{\boldsymbol{s}}\otimes (s_i+E_{ji})(s_j-E_{jj}) v; \\
(x_ix_j \partial_i)\partial_j(\boldsymbol{x}^{\boldsymbol{s}}\otimes v_{\lambda})&=&s_i(s_j-\lambda_j)\cdot \boldsymbol{x}^{\boldsymbol{s}}\otimes v_{\lambda},
\end{eqnarray*}
where $\boldsymbol{\lambda} = (\lambda_1,...,\lambda_n)$. Thus 
$$(s_i+E_{ji})(s_j-E_{jj})v=s_i(s_j-\lambda_j)v.$$
Using the minimality of $\boldsymbol{\eta}_1$, after taking the $\boldsymbol{\eta}_1$-components of the vector above, we obtain $E_{jj}v_1 = \lambda_j v_1$ for all $j<n$. Thus $\boldsymbol{\lambda}-\boldsymbol{\eta}_1 = c \vareps_n$ some $c \in \C$. But since $\boldsymbol{\eta}_1$ is in the support of $L(\boldsymbol{\lambda}')$ and $\boldsymbol{\lambda}'$ is a maximal weight in this support, we have $\boldsymbol{\lambda}'-\boldsymbol{\lambda} \geq - c \vareps_n$. With similar reasoning we obtain $\boldsymbol{\lambda}-\boldsymbol{\lambda}' \geq - c' \vareps_n$ for some $c' \in \C$. Thus $\boldsymbol{\lambda}-\boldsymbol{\lambda}' \in \C \vareps_n$. Now using this and the Weyl dimension formula for $\dim L(\boldsymbol{\lambda}) = \dim L(\boldsymbol{\lambda}')$, we prove that  $\boldsymbol{\lambda} = \boldsymbol{\lambda}'$.
\end{proof}
In what follows we introduce some important subquotients of the  $W_n$-modules $T(\boldsymbol{\nu},\boldsymbol{\lambda}) $ defined above. First, for any ${\bf z} \in {\mathbb C}^n$, we set $\mathcal{PM}({\bf z}) = \{+,-\}^{{\rm Int} (\bf z)}$. Every element $J: {\rm Int} (\bf z) \to \{ +,-\}$ of $\mathcal{PM}({\bf z}) $ will be written as $({i_1}^{J(i_1)},...,{i_k}^{J(i_k)})$, where ${\rm Int} ({\bf z})  = \{ i_1,...,i_k\}$ and $i_1<\cdots <i_k$. For example, if ${\rm Int} ({\bf z})  = \{1, 2\}$, then the elements of $\mathcal{PM}({\bf z})$ are: $(1^+, 2^+)$, $(1^+, 2^-)$, $(1^-, 2^+)$, $(1^-, 2^-)$. 

For every element $J$ in  $\mathcal{PM}({\bf z})$ we write $J^+$ (respectively,  $J^-$) for the subset of $J$ consisting of all $i_j^{J(i_j)}$ with $J(i_j) = ``+''$  (respectively,  $J(i_j) = ``-''$).

\begin{definition} \label{def-tensor} Let $\boldsymbol{\lambda}$ be a dominant integral $\mathfrak{gl}_n$-weight,  $\boldsymbol{\nu} \in {\mathbb C}^n$, and $J$  be in $\mathcal{PM} (\boldsymbol{\lambda} - \boldsymbol{\nu})$. We define $T(\boldsymbol{\nu}, \boldsymbol{\lambda}, J)$ as follows.

\begin{itemize}
\item[(i)] $T(\boldsymbol{\nu}, \boldsymbol{\lambda}, \emptyset) := T(\boldsymbol{\nu}, \boldsymbol{\lambda})$.

\item[(ii)]  If $J^+ \neq \emptyset$ and $J^- = \emptyset$, then
$$
T(\boldsymbol{\nu}, \boldsymbol{\lambda}, J) := \Span \{ x^{\eta} \otimes v_{\mu} \; | \; v_{\mu} \in L_{\mathfrak{gl}} (\lambda)^{\mu}, \eta_i - \mu_i \in {\mathbb Z}_{\geq 0} \mbox{ for all } i \in J\}. 
$$

\item[(iii)] If  $J^- \neq \emptyset$, then 
$$
T(\boldsymbol{\nu}, \boldsymbol{\lambda}, J) := T(\boldsymbol{\nu}, \boldsymbol{\lambda}, J^+)/\left( \sum_{j^- \in J^-}T(\boldsymbol{\nu}, \boldsymbol{\lambda}, J^+ \cup \{j^+\}) \right)
$$
\end{itemize}
\end{definition}
It is easy to check that if $J^- = \emptyset$, then $T(\boldsymbol{\nu}, \boldsymbol{\lambda}, J)$ is a submodule of $T(\boldsymbol{\nu}, \boldsymbol{\lambda})$. Therefore all  $T(\boldsymbol{\nu}, \boldsymbol{\lambda}, J)$ are bounded weight $W_n$-modules of degree $\dim L_{\mathfrak{gl}} (\boldsymbol{\lambda})$.

Using (\ref{x-i-action}), we easily endow $T(\boldsymbol{\nu}, \boldsymbol{\lambda}, J)$ with an $\A_n$-module structure and the resulting module $T(\boldsymbol{\nu}, \boldsymbol{\lambda}, J, c)$ has central charge $c$. In what follows, we will call both $T(\boldsymbol{\nu}, \boldsymbol{\lambda}, J)$ and $T(\boldsymbol{\nu}, \boldsymbol{\lambda}, J, c)$ \emph{generalized tensor modules}, or simply \emph{tensor modules}. 


\subsection{Tensor $\A_1$- and $W_1$-modules and classification of simple weight $W_1$-modules with finite weight multiplicities}
In the case $n=1$ we will use the notation $T(\nu, \lambda)$ and  $T(\nu, \lambda, c)$  for the tensor modules corresponding to $\nu, \lambda \in {\mathbb C}$. More explicitly,  $T(\nu, \lambda, c) = \Span \{ x^{\nu + \ell} \otimes v_{\lambda} \; | \; \ell \in {\mathbb Z}\}$, with action of $\A$ defined by:
\begin{eqnarray*}
D_i (x^{\nu + \ell} \otimes v_{\lambda}) & = &  (\nu + \ell + i \lambda) x^{\nu + \ell + i} \otimes v_{\lambda},\\
I_j (x^{\nu + \ell} \otimes v_{\lambda}) & = &  c x^{\nu + \ell + j} \otimes v_{\lambda}
\end{eqnarray*}

\begin{remark}
Another important class of modules that appears in the literature consists of the \emph{tensor densities modules}. Namely these are the $\A$-modules $\mathcal F (\nu , \lambda, c) = x^{\mu} {\mathbb C}[x^{\pm 1}] (dx)^{\lambda}$ of central charge $c$ and with the natural action of the generators $D_i$ and $I_j$. One easily can show that $\mathcal F (\nu, \lambda, c) \simeq T(\nu + \lambda, \lambda, c)$.
\end{remark}
We will also write $T(\lambda, \lambda, +) = T(\lambda, \lambda, 1^+)$ and  $T(\lambda, \lambda, -) = T(\lambda, \lambda, 1^-)$. In particular,   $T(\lambda, \lambda, +) = \Span \{ x^{\lambda + n} \otimes v_{\lambda}\; | \; n \in {\mathbb Z}_{\geq 0}\}$ and  $T(\lambda, \lambda, -)  = T(\lambda, \lambda) / T(\lambda, \lambda, +)$. Furthermore, the corresponding $\A$-modules to $T(\lambda, \lambda, \pm)$ will be denoted by  $T(\lambda, \lambda, c, \pm)$. The Jordan-H\"older decomposition of  the modules $T(\nu, \lambda)$ and $T(\nu, \lambda,c)$ is described in the following two propositions. The proof is standard and is omitted.
\begin{proposition} \label{prop-tens-w1-mod} Let $\lambda, \nu \in {\mathbb C}$
\begin{itemize}
\item[(i)] The $W_1$-modules $T(\nu_1, \lambda_1)$ and $T(\nu_2, \lambda_2)$ are isomorphic if and only if:
\begin{itemize}
\item[(a)] $\nu_1-\nu_2 \in {\mathbb Z}$ and $\lambda_1 = \lambda_2$, or 
\item[(b)] $\nu_1-\nu_2 \in {\mathbb Z}$, $\nu_1 \notin \Z$, and $\{\lambda_1, \lambda_2\} = \{ 0,1\} $. 
\end{itemize}
\item[(ii)] The $W_1$-module $T(\nu, \lambda)$ is simple if and only if $\lambda - \nu \notin {\mathbb Z}$.
\item[(iii)]   The $W_1$-module $T(\lambda, \lambda, +)$ is simple if and only if $\lambda \neq 0$.  The $W_1$-module $T(0, 0, +)$ has length two with a simple submodule isomorphic to $\mathbb C$ and a simple quotient isomorphic to $T(1,1,+)$.
\item[(iv)]   The $W_1$-module $T(\lambda, \lambda, -)$ is simple if and only if $\lambda \neq 1$.  The $W_1$-module $T(1, 1, -)$ has length two with a simple submodule isomorphic to $T(0,0,-)$, and a simple quotient isomorphic to $\mathbb C$.
\end{itemize}
\end{proposition}

In view of the above proposition, for $\lambda \neq 0$ we set $L(\lambda,+):=T(\lambda, \lambda, +)$ and $L(\lambda, -) = T(\lambda+1, \lambda+1, -)$. We also set $L(0) := \mathbb C$. Note that $L(\lambda, +)$  and $L(\lambda, -)$  are highest weight modules with highest weight $\lambda$ with respect to the Borel subalgebras ${\mathfrak b} (+) = \Span \{ D_i \; | \; i \in  {\mathbb Z}_{\geq 0} \}$ and ${\mathfrak b} (-) = \Span \{ D_i \; | \; i \in  \{0,-1\}\}$, respectively.

We now look at the structure of the tensor $\A$-modules.  In the case $c = 0$, we simply restate Proposition \ref{prop-tens-w1-mod} replacing the statements for $T(\nu,\lambda)$,  $T(\lambda ,\lambda,\pm)$ by $T(\nu,\lambda,0)$, $T(\lambda ,\lambda,0,\pm)$, respectively. We also write  $L(\lambda,0,+)$ and $L(\lambda,0,-)$ for  $T(\lambda,\lambda,0,+)$ and $T(\lambda+1,\lambda+1,0,-)$ if $\lambda \neq 0$, and $L(0,0) = \C$.

For tensor $\A$-modules with nonzero central charge we have the following.

\begin{proposition} \label{prop-tens-a1-mod} Let $\lambda, \nu \in {\mathbb C}$ and let $c \neq 0$.
\begin{itemize}
\item[(i)] The $\A$-modules $T(\nu_1, \lambda_1,c)$ and $T(\nu_2, \lambda_2,c)$ are isomorphic if and only if $\nu_1-\nu_2 \in {\mathbb Z}$ and $\lambda_1 =  \lambda_2$. 
\item[(ii)] The $\A$-module $T(\nu, \lambda,c)$ is simple if and only if $\lambda - \nu \notin {\mathbb Z}$.
\item[(iii)]   The $\A$-modules $T(\lambda, \lambda, c, +)$ and $T(\lambda, \lambda, c, -)$ are simple for all $\lambda \in \C$.
\end{itemize}
\end{proposition}

Naturally, for $c \neq 0$ we set $L(\lambda,c,+) := T(\lambda,\lambda,c,+)$ and $L(\lambda,c,-) := T(\lambda+1,\lambda+1,c,+)$.

We finish this subsection with the classification theorem for all simple weight $W_1$-modules with finite weight multiplicities due to  O. Mathieu, see \cite{M-Vir}.

\begin{theorem} \label{th-class-w1}
Every simple weight $W_1$-module with finite weight multiplicities is isomorphic to a module in the following list:  $T(\nu, \lambda)$, $\lambda -\nu \notin {\mathbb Z}$, $L(\eta, +)$, $L(\eta, -)$, $\eta \neq 0$, $L(0)$. The only isomorphisms among the modules in the list are: $T(\nu, \lambda) \simeq T(\nu + n, \lambda)$ for $n \in {\mathbb Z}$ and $\lambda - \nu \notin {\mathbb Z}$; $T(\nu,0) \simeq T(\nu,1)$, for $\nu \notin \Z$.
\end{theorem}

\subsection{Twisted localization of weight modules}
We first introduce the twisted localization functor in a general setting. Let $\mathcal U$, $\mathcal H = \C [\h]$ be as in \S \ref{subsec-wht}. Let  $a$ be an ad-nilpotent element of $\mathcal U$. Then the set $\langle a \rangle = \{ a^n \; | \; n \geq 0\}$ is an Ore subset of $\mathcal U$ (see for example \S 4 in \cite{M}) which allows us to define the  $\langle a \rangle$-localization $D_{\langle a \rangle} \mathcal U$ of $\mathcal U$. For a $\mathcal U$-module $M$  by $D_{\langle a \rangle} M = D_{\langle a \rangle} {\mathcal U} \otimes_{\mathcal U} M$ we denote the $\langle a \rangle$-localization of $M$. Note that if $a$ is injective on $M$, then $M$ is isomorphic to a submodule of $D_{\langle a \rangle} M$. In the latter case we will identify $M$ with that submodule, and will consider $M$ as a submodule of $D_{\langle a \rangle} M$.

We next recall the definition of the generalized conjugation of $D_{\langle a \rangle} \mathcal U$ relative to $x \in {\mathbb C}$. This is the automorphism  $\phi_x : D_{\langle a \rangle} \mathcal U \to D_{\langle a \rangle} \mathcal U$ given by $$\phi_x(u) = \sum_{i\geq 0} \binom{x}{i} \ad (a)^i (u) a^{-i}.$$ If $x \in \Z$, then $\phi_x(u) = a^xua^{-x}$. With the aid of $\phi_x$ we define the twisted module $\Phi_x(M) = M^{\phi_x}$ of any  $D_{\langle a \rangle} \mathcal U$-module $M$. Finally, we set $D_{\langle a \rangle}^x M = \Phi_x D_{\langle a \rangle} M$ for any $\mathcal U$-module $M$ and call it the \emph{twisted localization} of $M$ relative to $a$ and $x$. We will use the notation $a^x\cdot m$  (or simply $a^x m$) for the element in  $D_{\langle a \rangle}^x M$ corresponding to $m \in D_{\langle a \rangle} M$. In particular, the following formula holds in $D_{\langle a \rangle}^{x} M$:
$$
u (a^{x} m) = a^{x} \left( \sum_{i\geq 0} \binom{-x}{i} \ad (a)^i (u) a^{-i}m\right)
$$
for $u \in \mathcal U$, $m \in D_{\langle a \rangle}  M$.

We easily check that if $a$ is a weight element of $\mathcal U$ (i.e. $a \in {\mathcal U}^{\mu}$ for some $\mu \in {\mathfrak h}^*$) and $M$ is a (generalized) $({\mathcal U}, {\mathcal H})$-weight module, then $D_{\langle a \rangle}^x M$ is a  (generalized) $({\mathcal U}, {\mathcal H})$-weight module.  We will apply the twisted localization functor for several pairs $({\mathcal U}, {\mathcal H})$, and in particular  in the following two cases:  $\mathcal U = U(\mathcal A)$, $\mathcal H = \C[\h_{\A}]$;  and  $\mathcal U = U(W_2)$, $\mathcal H= \C[\h_{W_2}]$. 

\begin{lemma} \label{lem-tw-loc-simple}
Let $a \in {\mathcal U}$ be an $\ad$-nilpotent weight element in ${\mathcal U}$, $M$ be a simple $a$-injective weight ${\mathcal U}$-module, and let $z \in \C$. If $N$ is any simple submodule of $D^z_{\langle a \rangle}M$, then $D_{\langle a \rangle}M \simeq D^{-z}_{\langle a \rangle}N$. In particular, if $a$ acts bijectively on $M$, $M  \simeq D^{-z}_{\langle a \rangle}N$.
\end{lemma}
\begin{proof}
We use the fact  that if $M$ is a simple weight ${\mathcal U}$-module, then $D_{\langle a \rangle}M$ and $D_{\langle a \rangle}^z M$  are simple $D_{\langle a \rangle} {\mathcal U}$-modules. So, if  $N$ is any simple submodule of $D^z_{\langle a \rangle}M$, then $D_{\langle a \rangle} N$ is a submodule of $D^z_{\langle a \rangle}M$. This forces $D_{\langle a \rangle} N \simeq D^z_{\langle a \rangle}M$ and $D^{-z}_{\langle a \rangle} N \simeq D_{\langle a \rangle}M$. If $a$ acts bijectively, then $M \simeq D_{\langle a \rangle}M$. \end{proof}

The above lemma will be applied both for $\mathcal U = U(\A)$ and  $\mathcal U = U(W_2)$. For reader's convenience, we list some (but not all) ad-locally nilpotent weight elements $a$ in $\mathcal U$ in these two cases:
 \begin{itemize}
 \item $a = D_{-1}, I_j$, $j \geq 0$ for $\mathcal U = U(\A)$;
  \item $a = \partial_1, \partial_2, x_1\partial_2, x_2\partial_1$ for $\mathcal U = U(W_2)$.
 \end{itemize}

\begin{lemma} \label{lem-iso-loc-a} Let $\nu \notin \Z$. Then the following $\A$-module isomorphisms hold. 
\begin{itemize}
\item[(i)] $D_{\langle I_1 \rangle}^{\nu} T (\lambda, \lambda, c, +) \simeq T (\lambda + \nu, \lambda, c)$, whenever $c \neq 0$. 
\item[(ii)] $D_{\langle D_{-1} \rangle}^{-\nu} T (\lambda, \lambda, c, -) \simeq T (\lambda + \nu, \lambda,c)$.
\item[(iii)] $D_{\langle D_{-1} \rangle}^{\nu-\eta} T (\lambda+ \nu, \lambda,c) \simeq  T (\lambda + \eta, \lambda,c)$, whenever  $ \eta \notin {\mathbb Z}$.
\end{itemize}
\end{lemma}
\begin{proof} 
 (i) Note that  if $c \neq 0$, the set $\{ I_1^{\nu + \ell} (x^{\lambda} \otimes v_{\lambda})\; | \; \ell \in {\mathbb Z} \}$ forms a basis of $D_{\langle I_1 \rangle}^{\nu} T (\lambda, \lambda,c, +)$. Then  it is easy to check that  the map 
$$ I_1^{\nu + \ell} (x^{\lambda} \otimes v_{\lambda}) \mapsto  c^{\nu + \ell} x^{\lambda + \nu + \ell} \otimes v_{\lambda}$$
provides an isomorphism $D_{\langle I_1 \rangle}^{\nu} T (\lambda, \lambda, c,+) \simeq T (\lambda + \nu, \lambda,c)$. 

(ii) Note that the set $\{ D_{-1}^{-\nu - \ell -1} (x^{\lambda-1} \otimes v_{\lambda} + T(\lambda, \lambda, c, +))\; | \; \ell \in {\mathbb Z} \}$ forms a basis of $D_{\langle D_{-1} \rangle}^{-\nu} T (\lambda, \lambda, c, -)$. Then it is easy to check that the map 
$$ D_{-1}^{-\nu - \ell -1} (x^{\lambda-1} \otimes v_{\lambda} + T(\lambda, \lambda, c, +)) \mapsto D_{-1}(\nu, \ell)x^{\lambda + \nu + \ell} \otimes v_{\lambda},$$
where $D_{-1}(\nu, \ell) = \nu (\nu - 1)...(\nu - (-\ell -1))$ if $l < 0$ and $D_{-1}(\nu, \ell)  = \frac{1}{(\nu+1)...(\nu+\ell)}$ if $\ell \geq 0$, provides an isomorphism $D_{\langle D_{-1} \rangle}^{-\nu} T (\lambda, \lambda, c, -) \simeq T (\lambda + \nu, \lambda,c)$.

Part (iii) easily follows from (ii) and the additive property of the twisted localization functors: $D_{\langle D_{-1} \rangle}^{\nu_1 + \nu_2} M \simeq D_{\langle D_{-1} \rangle}^{\nu_1} D_{\langle D_{-1} \rangle}^{\nu_2} M$. \end{proof}
\begin{remark}
Note that the isomorphism in  Lemma \ref{lem-iso-loc-a}(ii)  holds even in the case  $(\lambda, c) = (1,0)$, namely when $ T (\lambda, \lambda, c, -)$ and $T (\lambda + \nu, \lambda,c)$ are not simple. Also, the coefficients $D_{-1} (\nu, \ell)$ in the  proof of Lemma \ref{lem-iso-loc-a}(ii) can also be defined in terms of Gamma-functions: $D_{-1}(\nu, \ell)  = \frac{\Gamma (\nu + 1)}{\Gamma (\nu + \ell + 1)}$.
\end{remark}

\subsection{Parabolic subalgebras of $W_2$ and Penkov-Serganova Parabolic Induction Theorem}

To keep the content of the paper short we will avoid the general definition of a parabolic subalgebra of $W_n$. Such a definition in terms of flags of real subspaces can be found in \S1 of \cite{PS}. Alternatively, one can use the general definition of a parabolic set of roots $P $ and then define a parabolic subalgebra $\mathfrak{p}_P$ associated with $P$ following \cite{GY}. The problem is that the root system of $W_n$ is neither symmetric nor finite. 

In what follows we fix $\sigma$ to be the automorphism of $\Delta_{W_2}$ interchanging $\varepsilon_1$ and $\varepsilon_2$. This automorphism naturally defines an automorphism of $\h_{W_2}^*$ , $\h_{W_2}$ and $W_2$. With a slight abuse of notation we will denote all resulting automorphisms by $\sigma$.

We have a natural embedding of $\mathfrak{sl}_{n+1}$ in $W_n$ arising from the infinitesimal action of the group $PSL (n+1)$ on ${\mathbb C}P^n$. In explicit terms, the embedding $\Phi$ is defined by $E_{ij} \mapsto x_i \partial_j$, $1\leq i,j \leq n$,  $E_{0k} \mapsto x_i {\mathcal E}$, $E_{k0} \mapsto -\partial_k$, $1\leq k \leq n$, where ${\mathcal E} = \sum_{k=1}^nx_k\partial_k$. 
With the embedding of $\mathfrak{sl}(n+1)$ in $W_n$ in mind we fix the Cartan subalgebra $\mathfrak{h}$ of $\mathfrak{sl}(n+1)$  to be the one corresponding to $\mathfrak{h}_W$ under the embedding $\Phi$. Again with a slight abuse of notation, the root system of $\mathfrak{sl}_{n+1}$ relative to $\mathfrak h$ will be denoted by $\Delta_{\mathfrak{s}} = \{ \varepsilon_i - \varepsilon_j \; | \; 0 \leq i \neq j \leq n\}$.  Since we will deal with parabolic subalgebras of $W_n$ that are induced from parabolic subalgebras of $\mathfrak{sl}_{n+1}$, we will limit out attention to this case only.

We now recall one of the few equivalent definitions of a parabolic subalgebra of $\mathfrak{sl}_{n+1}$. 
 A \emph{parabolic subset of roots} of $\mathfrak{s} = \mathfrak{sl}_{n+1}$ is a proper subset $P_\mathfrak{s}$ of $\Delta_{\mathfrak{s}} $ for which
$$
\Delta_{\mathfrak{s}}  = P_\mathfrak{s} \cup(-P_\mathfrak{s}) \; \; \; \mbox{and} \; \; 
\alpha, \beta \in P_\mathfrak{s} \; {\text { with }} \; \alpha + \beta \in \Delta_\mathfrak{s} 
\; \; 
{\text { implies }} \; \; \alpha + \beta \in P_\mathfrak{s}.
$$
For a parabolic subset of roots 
$P_\mathfrak{s}$ of $\Delta_{\mathfrak{s}} $, we call $L_\mathfrak{s} := P_\mathfrak{s} \cap (-P_\mathfrak{s})$ 
the {\em{Levi component}} of $P_\mathfrak{s}$, $N_\mathfrak{s}^+ := P_\mathfrak{s} \backslash (-P_\mathfrak{s})$ the 
{\em{nilradical}} of $P_\mathfrak{s}$, and $P_\mathfrak{s} = L_\mathfrak{s} \sqcup N_\mathfrak{s}^+$ the Levi decomposition 
of $P$. We call
$$
\mathfrak{p}_{P_\mathfrak{s}} := \mathfrak{h} \oplus \left( \bigoplus_{\alpha \in P_\mathfrak{s}} \, \mathfrak{s}^\alpha \right).
$$
a parabolic subalgebra of $\mathfrak{s}$ associated with $P_\mathfrak{s}$. 

If $P_{\mathfrak{s}}= L_{\mathfrak{s}} \sqcup N_{\mathfrak{s}}^+ $ is a parabolic subset of roots of $\mathfrak s$, then we call $P= L \sqcup N^+ $ \emph{a parabolic subset of roots of $W_n$ induced from $P_{\mathfrak s}$}, where $L:= {\mathbb Z} L_{\mathfrak{s}} \cap \Delta_W$ and $N^+ := \left( {\mathbb Z}_{\geq 0}P_{\mathfrak s} \cap \Delta_W \right) \setminus L$. If $P_{\mathfrak{s}}= L_{\mathfrak{s}} \sqcup N_{\mathfrak{s}}^+ $ is a parabolic subset of roots of $\mathfrak s$, and $P = L \sqcup N^+$ is the parabolic subset of roots of $W_n$ induced from $P_{\mathfrak s}$, we set $N_{\mathfrak s}^{-} = \Delta_{\mathfrak s} \setminus P_{\mathfrak s}$ and $N^{-} = \Delta_W \setminus P$. The reader is referred  to Lemma 3 in \cite{PS} for a proof of the fact that every parabolic subset of roots of $W_n$ induced from one of $\mathfrak s$ is indeed a parabolic subset of roots of $W_n$. We call $P^- = L \sqcup N^-$ the \emph{opposite to $P$} parabolic subset. Analogously to the case of $\mathfrak{sl}_{n+1}$ we define 
$$
\mathfrak{p}_{P} := \mathfrak{h} \oplus \left( \bigoplus_{\alpha \in P} \, \mathfrak{s}^\alpha \right).
$$
to be the parabolic subalgebra of $W_n$ associated with $P$ (or with $P_{\mathfrak s}$). By $\mathfrak p^-$ we denote the parabolic subalgebra associated with $P^-$ and call it the opposite to $\mathfrak p$ parabolic subalgebra.

If $P_{\mathfrak s} = L_{\mathfrak{s}} \sqcup N_{\mathfrak{s}}^+ $ is a parabolic subset of roots of $\mathfrak s$ for which $L_{\mathfrak{s}}  = \emptyset$, we will call the corresponding subalgebras $\mathfrak{p}_{P_\mathfrak{s}} $ and $\mathfrak{p}_{P}$ \emph{Borel subalgebras} of $\mathfrak{sl}_{n+1}$ and $W_n$, respectively.

Below we list all parabolic subsets of roots of $W_2$ induced from parabolic subsets of roots of $\Delta_{\mathfrak{s}}$ of $\mathfrak{sl} (3)$ together with the corresponding parabolic subalgebras.

\begin{example} \label{ex-par-w2} For a subset $J$ of the real vector space ${\mathbb R} \Delta_W$, let $P(J) = \{\alpha \in \Delta_W \; | \; (\alpha, s) \in {\mathbb R}_{\leq 0} \mbox{ for every } s\in J\}$. Then all parabolical subset of roots of $W_2$ induced by parabolic subsets of roots of $\mathfrak{sl} (3)$ are of the form $P(J)$ for a set $J$ consisting of one or two elements. All possible sets $J$ together with the notation that will be used for the corresponding parabolic subset of roots $P(J)$ and the parabolic subalgebra $\mathfrak p (J) = \mathfrak{p}_{P(J)}$ are listed below.
\begin{itemize}
\item[(i)] $J = \{ \varepsilon_1 - \varepsilon_0\}$, $P(1^+)$, ${\mathfrak p}(1^+)$.

\item[(ii)] $J = \{ \varepsilon_0 - \varepsilon_1\}$, $P(1^-)$, ${\mathfrak p}(1^-)$.

\item[(iii)] $J = \{ \varepsilon_2 - \varepsilon_0\}$, $P(2^+)$, ${\mathfrak p}(2^+)$.

\item[(iv)] $J = \{ \varepsilon_0 - \varepsilon_2\}$, $P(2^-)$, ${\mathfrak p}(2^-)$.

\item[(v)] $J = \{ \varepsilon_1 + \varepsilon_2\}$, $P(12^+)$, ${\mathfrak p}(12^+)$.

\item[(vi)] $J = \{ -\varepsilon_1 - \varepsilon_2\}$, $P(12^-)$, ${\mathfrak p}(12^-)$.

\item[(vii)] $J = \{\varepsilon_1 - \varepsilon_0, \varepsilon_0 - \varepsilon_2\}$, $P(1^+,2^-)$, ${\mathfrak p}(1^+,2^-)$.

\item[(viii)] $J = \{\varepsilon_0 - \varepsilon_1, \varepsilon_2 - \varepsilon_0\}$, $P(1^-,2^+)$, ${\mathfrak p}(1^-,2^+)$.

\item[(ix)] $J = \{\varepsilon_0 - \varepsilon_2, - \varepsilon_1 - \varepsilon_2\}$, $P(2^-,12^-)$, ${\mathfrak p}(2^-,12^-)$.

\item[(x)] $J = \{\varepsilon_1 - \varepsilon_0,  \varepsilon_1  + \varepsilon_2\}$, $P(1^+,12^+)$, ${\mathfrak p}(1^+,12^+)$.

\item[(xi)] $J = \{\varepsilon_0 - \varepsilon_1, - \varepsilon_1 - \varepsilon_2\}$, $P(1^-,12^-)$, ${\mathfrak p}(1^-,12^-)$.

\item[(xii)] $J = \{\varepsilon_2 - \varepsilon_0,  \varepsilon_1  + \varepsilon_2\}$, $P(2^+,12^+)$, ${\mathfrak p}(2^+,12^+)$.

\end{itemize}

\end{example}

\begin{remark}
Note for example that $P(1^+) =  {\mathbb Z}_{\leq 0} \varepsilon_1 + {\mathbb Z} \varepsilon_2$ which seems counterintuitive. This sign change is imposed in order to match the notation of the parabolic subalgebras  with the notation of the corresponding induced modules, see for example Proposition \ref{prop-w2-bnd-nec}.
\end{remark}

The following parabolic induction theorem follows from Lemma 11 in \cite{PS}.
\begin{theorem} \label{th-par-ind}
Let $M$ be a simple weight $W_n$-module with finite weight multiplicities. Then there is a parabolic subalgebra $\mathfrak p = \mathfrak{l} \oplus \mathfrak{n}^+$ of $W_n$ induced from a parabolic subalgebra of $\mathfrak s = \mathfrak{sl}_{n+1}$ and a simple $\mathfrak{p}$-module $S$ with a trivial action of $ \mathfrak{n}^+$ such that $M$ is a quotient of the induced module $U(W_n) \otimes_{U(\mathfrak p)} S$. Moreover, there is $\lambda \in \Supp L$ such that $\lambda + \alpha \notin \Supp M$ for all $\alpha$ in $N^+ = \Delta_{\mathfrak{n}^+}$.
\end{theorem}

\subsection{Supports of $W_2$-modules}

In what follows we describe all possible supports of simple weight $W_2$-modules with finite weight mutliplicities, see Example 2 in \cite{PS}.

\begin{proposition} \label{prop-w2-par}
All possible supports of simple weight $W_2$-modules with finite weight multiplicities are exactly in  one of the following forms: 
\begin{itemize}
\item[(i)--(xii)] $\lambda + P(J)$ for $J$ being one in the list {\rm (i)-(xii)} of Example \ref{ex-par-w2}, 
\item[(xiii)] $\left(\lambda + P(2^-,12^-) \right) \cap \left(\sigma (\lambda) + P(1^-,12^-) \right) $,
\item[(xiv)] $\left(\lambda + P(1^+,12^+) \right) \cap \left(\sigma (\lambda) + P(2^+,12^+) \right) $,
\item[(xv)] $\lambda + {\mathbb Z}^2$,
\item[(xvi)] $\{ 0\}$,
\end{itemize}
with the conditions on $\lambda$ as follows: no restrictions on $\lambda$ in cases {\rm (i)-(vi), (xv)}; $\lambda \neq 0$ in cases {\rm (vii)-(viii)}; $\lambda_2 - \lambda_1 \notin {\mathbb Z}_{\geq 0}$ in cases {\rm (ix)-(x)}; $\lambda_1 - \lambda_2 \notin {\mathbb Z}_{\geq 0}$ in cases {\rm (xi)-(xii)}; $\lambda_1 - \lambda_2 \in {\mathbb Z}_{\geq 0}$, $\lambda \neq 0$ in cases {\rm (xiii)-(xiv)}.
\end{proposition}

A simple weight module $M$ of type (xv), i.e. such that $\Supp M = \lambda + \Z^2$, will be called a \emph{dense module}.

\section{Classification of simple bounded $\A$-Modules}

We start with a general property of the bounded $\A_n$-modules.
\begin{proposition} \label{prop-fin-len}
Every bounded $\A_n$-module  (and, hence  $W_n$-module) whose support is a subset of $\lambda + {\mathbb Z}^n$ for some $\lambda$ has finite length. 
\end{proposition}
\begin{proof} This follows from the fact that $\A_n$ has a subalgebra isomorphic to $\mathfrak{sl}_{n+1}$ and that the statement holds for bounded $\mathfrak{sl}_{n+1}$-modules with the same support property, see Lemma 3.3 in \cite{M}.
\end{proof}

In the rest of this section we work with $n=1$, i.e. with $\A$.
\begin{lemma} \label{lem-loc-fin}
Let $M$ be a simple weight $\A$-module of central charge $c$. If $D_{-1}$ acts finitely on $M$, then either $M \simeq  L (\lambda , c, +)$ for some $\lambda, c$,  $(\lambda, c) \neq (0,0)$, or $M \simeq  L (0,0)$.
\end{lemma}
\begin{proof} Let $v \in M^{\lambda}$ be such that $D_{-1} v = 0$. Then $M = U(\A)v $ is an $\A^-$-highest weight module. Thus $U( \A) v$ is the unique simple  quotient of the induced module\\ $U(\A) \otimes _{U(\A ^- \oplus \A^0)} \C_{(\lambda, c)}$, where $\C_{(\lambda, c)}$ is the 1-dimensional $\A^0$-module of weight $(\lambda,c)$ on which $\A^-$ acts trivially. However, we know that $L(\lambda , c, +)$ and $L (0 , 0)$ are  such simple highest weight modules. 
\end{proof}

\begin{theorem} \label{th-bounded-a}
Let $M$ be a  simple bounded $\A$-module of central charge $c$. If $c = 0$, then $M$ is a simple bounded $W_1$-module, i.e. it is isomorphic to one of the modules listed in Theorem \ref{th-class-w1} with trivial action of $I_k$, $k\geq 0$. If $c\neq 0$, then  $M$ is isomorphic to one of the following:  $T (\nu , \lambda, c)$ ($\lambda - \nu  \notin {\mathbb Z}$), $L(\lambda, c, +)$, $L (\lambda, c, -)$. The only isomorphisms of the  listed $\A$-modules  for $c\neq 0$ are: $ T (\nu , \lambda, c) \simeq  T(\nu +n , \lambda, c)$, $n \in {\mathbb Z}$, for $\nu \notin \Z$.
 \end{theorem}

\begin{proof} By Lemma \ref{lem-loc-fin} we know that the result holds if $D_{-1}$ acts  finitely on $M$. So, for the rest of the proof,  we can assume that $D_{-1}$ acts injectively on $M$. We split the proof in two parts depending on the central charge $c$. In all statements we assume that $M$ is a simple bounded $\A$-module of central charge $c$.

\bigskip
\noindent{\it Case 1: Nonzero central charge, i.e. $c \neq 0$.} 

We split this case into two subclasses depending on the action of $I_1$ on $M$.

\begin{lemma} \label{lem-1} If $c \neq 0$,  $D_{-1}$ acts injectively on $M$, and $I_1$ acts finitely on $M$, then $M \simeq L(\lambda , c, -)$ for some $\lambda$.
\end{lemma}

\textit{Proof of Lemma \ref{lem-1}.} Let $\mathfrak a = \Span \{ D_{-1}, I_0, I_1\}$. Note that $\mathfrak a$  is a Lie subalgebra of $\A$ which is isomorphic to the three-dimensional Heisenberg Lie algebra. Furthermore, each weight space $M^{\lambda}$, $\lambda \in \h_{\A}^*$, is $I_1D_{-1}$-invariant, so $M$ considered as $\mathfrak a$-module is a generalized weight $(\mathfrak a, \h_{\mathfrak a})$-module for $ \h_{\mathfrak a} = \Span \{ I_1D_{-1}\}.$

The classification of simple generalized weight $(\mathfrak a, \h_{\mathfrak a})$-modules with nonzero central charge $c$ (equivalently, generalized weight modules of the Weyl algebra $U(\mathfrak a)/(I_0 - c)$) on which $I_1$ acts finitely is well-known. All such modules are simple weight $(\mathfrak a, \h_{\mathfrak a})$-modules  and are isomorphic to the module $\C[D_{-1}]$, such that $D_{-1} (D_{-1}^k) = D_{-1}^{k+1}$, $I_0 (D_{-1}^k)= c D_{-1}^k $, and $I_1 (D_{-1}^k) = - c k D_{-1}^{k-1}$ (see for example \S2 in \cite{GS1}). Note that if $c=0$ then we should add the trivial module ${\mathbb C}$ in that list, but this case is addressed in Case 2 below.

Let $d$ be the degree of $M$. Looking at the $\A$-support of $M$ we see that $M$ can not have more than $d$  simple ${\mathfrak a} $-subquotients. Indeed, if the converse is true, all such subquotients will be isomorphic as ${\mathfrak a} $-modules to $\C[D_{-1}]$, and then we can easily find an $\A$-weight space of $M$ of dimension bigger than $d$. In particular, $M$ has finitely many simple ${\mathfrak a}$-suqbquotients $M_i$ and $M_i = \Span \{ D_{-1}^k m_i \; | \; k \geq 0\}$ for some $m_i \in M$. Hence the $\A$ support of $M$ is bounded from the right, i.e. there is $\lambda \in \Supp M$ such that  $\mbox{Supp} M \subset \lambda + {\mathbb Z}_{\leq 0}$. Therefore $M$ is a simple $\A^+$-highest weight module whose highest weight is $\lambda$, that is $M \simeq L(\lambda , c, -)$. 

\medskip

\begin{lemma} \label{lem-2} Let $c \neq 0$ and  let both $D_{-1}$ and $I_1$ act injectively on $M$. 
\begin{itemize}
\item[(i)]There are $\nu \in \C$ and a simple $\A$-module $N$ on which $D_{-1}$ acts finitely, such that $M \simeq D_{\langle I_1 \rangle}^\nu N$.
\item[(ii)]  $M \simeq T (\nu, \lambda , c)$ for some $\lambda \in \C$.
\end{itemize}
\end{lemma}

\textit{Proof of Lemma \ref{lem-2}.}  First note that since $D_{-1}$ and $I_1$ act injectively on $M$, then $I_1$ acts bijectively on $M$. In particular, $D_{\langle I_1 \rangle} M \simeq M$. Now consider  $D_{\langle I_1 \rangle}^{-\nu} M$, for any $\nu \in \C$. Let $\lambda \in \mbox{Supp} M$. Then  $I_1 D_{-1} |_{M^\lambda}$  is an endomorphism on the finite-dimensional vector space $M^\lambda$. Let $\alpha$ be an eigenvalue of this endomorphism and let $I_1 D_{-1}v = \alpha v$ for $v \in M^{\lambda}$. Then
\begin{align*}
D_{-1} (I_1 ^{-\nu} v) = & I_1^{-\nu} \left ( \sum _{i \geq 0} \binom{\nu}{i}  \left (\ad I_1\right)^i (D_{-1})I_1^{-i} (v) \right ) \\
= & I_1 ^{-\nu } \left ((D_{-1} + \nu I_0 I_{1}^{-1}) (v)  \right ) \\
= & I_1 ^{-\nu -1 } \left ((\alpha+ \nu c ) v\right ) 
\end{align*} We first note that since both $D_{-1}$ and $I_1$ act injectively on $M$ and the weight space of $M$ are finite dimensional, then $I_1$ (and $D_{-1}$) act bijectively on $M$, hence $M \simeq D_{\langle I_1 \rangle} M$. 
If $\nu = - \frac{\alpha }{c}$, then $D_{-1} (I_1 ^\nu m) = 0 $. The elements of $D_{\langle I_1 \rangle}^{-\nu} M$ on which $D_{-1}$ acts  finitely form a submodule $N'$ of $D_{\langle I_1 \rangle}^{-\nu} M$. Then by Proposition \ref{prop-fin-len},  $N'$ has finite length, so we can choose a  simple $\A$-submodule $N$ of $N'$.  Then by  Lemma \ref{lem-tw-loc-simple}, $M \simeq D_{\langle I_1 \rangle}^\nu N$ which proves part (i). 

To prove (ii), we apply  Lemma \ref{lem-loc-fin} and obtain that $N \simeq L (\lambda, c, +)$. Therefore $M \simeq D_{\langle I_1 \rangle} M \simeq D_{\langle I_1 \rangle}^{\nu} L (\lambda, c,+) \simeq T (\lambda+ \nu, \lambda,  c)$. The last isomorphism follows from Lemma \ref{lem-iso-loc-a}(i).

\bigskip
\noindent{\it Case 2: Zero central charge, i.e. $c = 0$.} 

In this case we have the following lemma.

\begin{lemma} \label{lem-3} Suppose that $c =0$ and $D_{-1}$ acts injectively on $M$. Then $I_k =0$ on $M$ for all $k\geq 0$. In particular, $M$ is a simple $W_1$-module, and thus is isomorphic to one of the modules $T(\nu, \lambda , 0)$ ($\lambda - \nu \notin \Z$), $L (\eta , 0,-)$, $\eta \neq 0$.
\end{lemma}
\textit{Proof of Lemma \ref{lem-3}.} Let $d$ be the degree of $M$, let $\lambda \in \mbox{Supp}\, M$ and consider the endomorphisms $S = D_{-1}^2 I_2$  and $T = D_{-1} I_1 |_{M^{\lambda}}$ on the vector space $M^{\lambda}$ of dimension at most $d$. Using that $D_{-1}$ and $I_1$ commute, we easily check that $[T,S] = 2T^2$ and $[T^N,S] = 2N T^{N+2}$. Therefore the trace of the endomorphism $T^N = \left[\frac{1}{2N-4}T^{N-2}, S\right]$ is zero for all $N>2$. But then the sum of the $N$-th powers, $N >  2$, of the eigenvalues of $T$ is zero and hence $T$ is nilpotent. Thus $T^d = 0$. But using again that $I_1$ and $D_{-1}$ commute we find that $I_1^d = 0$ on $M$. Fix $N_0>0$ such that $I_1^{N_0} = 0$ and $I_1^{N_0-1} \neq 0$ on $M$. Let $v_0 \in M$ be such that   $I_1^{N_0-1}(v_0) \neq 0$. Then for $k\geq 1$, we have
$$
0 = D_{k-1}(I_1^{N_0}(v_0)) = I_1^{N_0} (D_{k-1}(v_0)) + N_0I_kI_1^{N_0-1}(v_0).
$$
Therefore $I_k(v) = 0$ for every $k\geq 1$ where $v = I_1^{N_0-1}(v_0)$. This implies that the set $M'$ of all $w$ with the property $I_kw =0$ for all $k\geq 0$ is nonzero. Since $M'$ is an $\A$-submodule of $M$, we have $M'=M$, which proves the first assertion of the lemma. The second part of the lemma follows from the classification of the simple weight $W_1$-modules, i.e. from Theorem \ref{th-class-w1}. 
\end{proof}

\section{Classification of simple bounded $W_2$-modules}

In this section we classify all simple bounded $W_2$-modules, i.e. weight $W_2$-modules with uniformly bounded set of weight multiplicities.

In what follows we  set $\mathfrak a := \Span \{ x_i\partial_j \; | \; 1\leq i,j\leq 2\}$. We know that the Cartan subalgebras of $W_2$, $\mathfrak a$ and $\mathfrak s$ coincide with $\mathfrak h_{W_2}$. Using this and the isomorphisms $\mathfrak s \simeq \mathfrak{sl}_3$, $\mathfrak a \simeq \mathfrak{gl}_2$ we will often write the weights of $\mathfrak{sl}_3$ and $\mathfrak{gl}_2$ as pairs $(\lambda_1,\lambda_2)$. In some cases, when  using  representation theory results for $\mathfrak{sl}_3$ we will write an  $\mathfrak{sl}_3$-weight $\lambda$ as $\lambda = \lambda_0 \varepsilon_0 + \lambda_1 \varepsilon_1 + \lambda_2 \varepsilon_2$ with $\lambda_0+\lambda_1+\lambda_2 = 0$ (in particular $\lambda = (\lambda_1,\lambda_2)$ as an element of $\h_{W_2}^*$).

We also set $W(x_i) := \Der (\C[x_i])$, $i=1,2$, $ \mathcal A (x_1) := W(x_1) \ltimes \left( \C[x_1] (x_2\partial_2)\right)$, and $ \mathcal A (x_2) := W(x_2) \ltimes \left( \C[x_2] (x_1\partial_1)\right)$. In particular, $\A (x_1) \simeq \A (x_2) \simeq \A$.

\begin{definition}
Let $J$ be a set from the list {\rm(i)-(xii)} of Example \ref{ex-par-w2}. We say that a simple weight $W_2$-module $M$ with finite weight multiplicities is \emph{of type $J$} if $M$ is the simple quotient of the module $U(W_2) \otimes_{U(\mathfrak p (J))} L$ for some simple $\mathfrak p (J)$-module $L$.
\end{definition}

We proceed with the classification of simple bounded $W_2$-modules $M$ in three steps depending on the type of $M$.

\subsection{Classification of simple bounded highest weight $W_2$-modules}
In this subsection we classify all bounded highest weight $W_2$-modules, namely all modules from cases (vii)--(xii) in the list of Example \ref{ex-par-w2} and Proposition \ref{prop-w2-par}. For simplicity, in this section, we will not use bold symbols for the vectors and multi-indexes. For example,  we write  $\lambda$  for $\boldsymbol{\lambda}$, etc.

Recall that for every $\lambda \in \C^2$, the tensor modules $T(\lambda,\lambda)$ has four subquotients $T(\lambda, \lambda, (1^+,2^+))$, $T(\lambda, \lambda, (1^+,2^-))$, $T(\lambda, \lambda, (1^-,2^+))$, $T(\lambda, \lambda, (1^-,2^-))$. Some important properties of the these four modules are collected in the next proposition. A proof is provided in \cite{Cav}  and is based on the description of the highest weight bounded $\mathfrak{sl}_3$-modules. 
\begin{proposition} \label{prop-tensor-simple-w2}
Let $\lambda \in \C^2$. 
\begin{itemize}
\item[(i)]  Let $J = (1^+,2^+)$. The $W_2$-module $T(\lambda, \lambda,  J)$ is simple if and only if  $\lambda \neq (0,0)$ and   $\lambda \neq (1,0)$. The $W_2$-module $T((0,0),(0,0),J)$ has length $2$ with simple submodule isomorphic to $\C$. The $W_2$-module $T((1,0),(1,0),J)$ has length $2$ with simple submodule isomorphic to $T((0,0),(0,0),J)/\C$ and simple quotient isomorphic to $T((1,1),(1,1),J)$.  

\item[(ii)] Let $J = (1^+,2^-)$ or $J = (1^-,2^+)$. The $W_2$-module $T(\lambda, \lambda,  J)$ is  simple if and only if $\lambda \neq (1,0)$. The module  $T((1,0), (1,0),  J)$ has length $3$ with simple submodule $T((0,0), (0,0),  J)$, simple quotient $T((1,1), (1,1),  J)$, and simple subquotient $\C$.

\item[(iii)]  Let $J = (1^-,2^-)$. The $W_2$-module $T(\lambda, \lambda,  J)$ is simple if and only if  $\lambda \neq (1,0)$ and $\lambda \neq (1,1)$. The $W_2$-module $T((1,0),(1,0),J)$ has length $2$ with simple submodule isomorphic to isomorphic to $T((0,0),(0,0),J)$.   The $W_2$-module $T((1,1),(1,1),J)$ has length $2$ with simple quotient isomorphic to $\C$ and simple submodule isomorphic to $T((1,0),(1,0),J)/T((0,0),(0,0),J)$.  
\end{itemize}

\end{proposition} \label{prop-char-for}

The character formulae of all tensor $W_2$-modules $T(\lambda, \lambda, J)$ follow directly from their definition. For a weight module $M$ with finite weight multiplicities, we write $\ch M = \sum_{\lambda \in \Supp M} \dim M^{\lambda} e^{\lambda}$.
\begin{proposition}
Let $\lambda \in \C^2$. Then the following identities hold.
\begin{itemize}

\item[(i)] $\displaystyle \ch T(\lambda,\lambda, (1^+, 2^+)) = \frac{\ch L_{\mathfrak{gl}}(\lambda)}{(1-e^{\varepsilon_1})(1-e^{\varepsilon_2})}$.

\item[(ii)] $\displaystyle \ch T(\lambda,\lambda, (1^+, 2^-)) = \frac{e^{-\varepsilon_2}\ch L_{\mathfrak{gl}}(\lambda)}{(1-e^{\varepsilon_1})(1-e^{-\varepsilon_2})}$,\;  $\displaystyle  \ch T(\lambda,\lambda, (1^-, 2^+)) = \frac{e^{-\varepsilon_1}\ch L_{\mathfrak{gl}}(\lambda)}{(1-e^{-\varepsilon_1})(1-e^{\varepsilon_2})}$.

\item[(iii)] $\displaystyle \ch T(\lambda,\lambda, (1^-, 2^-)) = \frac{e^{-\varepsilon_1-\varepsilon_2}\ch L_{\mathfrak{gl}}(\lambda)}{(1-e^{-\varepsilon_1})(1-e^{-\varepsilon_2})}$.

\end{itemize}
In particular, the degrees of all four modules equal $\dim L_{\mathfrak{gl}}(\lambda) = \lambda_1 - \lambda_2 + 1$.
\end{proposition}

For any Borel subalgebra $\mathfrak b$ of $W_2$ induced by a Borel subalgebra $\mathfrak b_{\mathfrak s}$ of ${\mathfrak s} \simeq \mathfrak{sl}_3$, by $L_{\mathfrak b}(\lambda)$ (respectively, by  $L_{\mathfrak b_{\mathfrak s}}^{\mathfrak{sl}}(\lambda)$) we denote the simple highest weight $W_2$-module (respectively, $\mathfrak s$-module) relative to $\mathfrak b$ (respectively, to ${\mathfrak b}_{\mathfrak s}$) with highest weight $\lambda$. In the case when $\mathfrak b_{\mathfrak s}$ is the standard Borel subalgebra $\mathfrak b_{st}$ of ${\mathfrak s} \simeq \mathfrak{sl}_3$, i.e. the one  with base $\Pi_{st} = \{ \varepsilon_0 - \varepsilon_1, \varepsilon_1 - \varepsilon_2\}$, we will write $L(\lambda)$ and $L^{\mathfrak{sl}}(\lambda)$ for $L_{\mathfrak b}(\lambda)$  and $L_{\mathfrak b_{\mathfrak s}}^{\mathfrak{sl}}(\lambda)$, respectively. Note that the Borel subalgebra of $W_2$ induced by $\mathfrak b_{st}$ is ${\mathfrak p} (2^+, 12^+)$. For $\mathfrak b_{st}^-$ (the opposite to the standard Borel subalegbra) and its induced Borel subalgebra ${\mathfrak p} (2^-, 12^-)$ of $W_2$, the corresponding modules will be denoted by $\tilde{L}(\lambda)$  and $\tilde{L}^{\mathfrak{sl}}(\lambda)$, respectively.

For a weight $W_2$-module $M = \bigoplus_{\lambda \in \h^*} M^{\lambda}$ with finite weight multiplicities, by $M^*$ we denote the restricted dual of $M$, namely the module $\bigoplus_{\lambda \in \h^*}  \Hom_{\mathbb C} (M^{\lambda}, {\mathbb C})$ with action defined by  $ (u f) (m) = f (-um)$. It is clear that $M^*$ is also a weight module with finite weight multiplicities. Moreover, $\left(L_{\mathfrak b} (\lambda)\right)^* \simeq L_{\mathfrak b^-}(-\lambda)$, where recall that ${\mathfrak b^-}$ is the opposite to ${\mathfrak b}$ Borel subalgebra. Certainly, the same isomorphism holds for the corresponding
highest weight $\mathfrak s$-modules, and Borel subalgebras  of  $\mathfrak{sl}_3$.

A weight $\lambda$ will be called \emph{$(W_2, \mathfrak b)$-bounded} (respectively, \emph{$(\mathfrak{sl}_3, \mathfrak b_{\mathfrak s})$-bounded}) if $L_{\mathfrak b}(\lambda)$ (respectively $L_{\mathfrak b_{\mathfrak s}}^{\mathfrak{sl}}(\lambda)$) is a bounded module. We will use the following classification of the $(\mathfrak{sl}_3, \mathfrak b_{\mathfrak s})$-bounded weights, see Lemma 7.1 in \cite{M}.

\begin{lemma} \label{lem-sl3-bw}
A weight $\lambda$ of $\mathfrak{sl}_3$ is $(\mathfrak{sl}_3, \mathfrak b_{\mathfrak s})$-bounded if and only if 
$(\lambda + \rho_{\mathfrak b_{\mathfrak s}}, \alpha) \in \Z_{\geq 0}$ for some root $\alpha$ of $\mathfrak b_{\mathfrak s}$, where $\rho_{\mathfrak b_{\mathfrak s}}$ is the half-sum of the ${\mathfrak b_{\mathfrak s}}$-positive roots of $\Delta_{\mathfrak{sl}_3}$.
\end{lemma}
Note that in the lemma above we may have more than one root $\alpha$ that satisfy the stated condition. In particular, if all three roots satisfy the condition, then $L_{\mathfrak b_{\mathfrak s}}^{\mathfrak{sl}}(\lambda)$ is finite dimensional.

\begin{lemma} \label{lem-hw-wbounded}
Let $\lambda \in \C^2$. Then $L(\lambda)$ is a bounded module if and only if $\lambda_1 - \lambda_2 \in \Z_{\geq 0}$.
\end{lemma}
\begin{proof}
If $\lambda_1 - \lambda_2 \in \Z_{\geq 0}$, by  Proposition \ref{prop-tensor-simple-w2} we know that  $L(\lambda)$ is a subquotient of $T(\lambda, \lambda, (1^+,2^+))$  (in fact,  $L(\lambda) \simeq T(\lambda, \lambda, (1^+,2^+))$ if $\lambda \neq (0,0), (1,0)$). Hence, $L(\lambda)$ is bounded. 

For the ``only if'' direction, we will prove the following equivalent statement: If $\tilde{L}(\mu)$ is bounded, then $\mu_2 - \mu_1 \in \Z_{\geq 0}$. The two statements are equivalent because $L(\lambda)^* = \tilde{L}(-\lambda)$. Since  $\mu$ is an $(\mathfrak{sl}_3, \mathfrak{b}^-_{st})$-bounded weight and $\rho_{\mathfrak b_{st}^-} = \varepsilon_2 - \varepsilon_0$, by Lemma \ref{lem-sl3-bw}, $\mu$ is one (or more than one) of the following three types:
 
\noindent {\it Type 1:} $\mu_2 - \mu_1 \in \Z_{\geq 0}$; 

\noindent {\it Type 2:}  $2 \mu_1 + \mu_2 \in \Z_{\geq 0}$; 

\noindent {\it Type 3:} $ \mu_1 + 2\mu_2 + 1\in \Z_{\geq 0}$.

Assume for the sake of contradiction that $\mu_2 - \mu_1 \notin \Z_{\geq 0}$, in particular, $\mu$ is of Type 2 or of Type 3. Then $\tilde{L} (\mu)$ is $\partial_1$-injective module. Indeed, if $\tilde{L} (\mu)$ is $\partial_1$-finite, then the $\A (x_1)$-module generated by a highest weight vector of  $\tilde{L} (\mu)$ must have finite support. But the only possible finite-dimensional $\A$-modules are the trivial modules, i.e. $\mu_1=\mu_2=0$,  contradicting  $\mu_2 - \mu_1 \notin \Z_{\geq 0}$.

Since  $\tilde{L} (\mu)$ is $\partial_1$-injective, it can be considered as a submodule of  $D_{\langle \partial_1\rangle }\tilde{L} (\mu)$. But then the quotient  $D_{\langle \partial_1\rangle }\tilde{L} (\mu)/ \tilde{L} (\mu)$ has a primitive vector relative to the Borel subalgebra $\mathfrak p (1^+, 2^-)$. Namely, this is the vector $\partial_1^{-1} v$ where $v$ is a highest weight vector of $ \tilde{L} (\mu)$. As a result $(\mu_1+1,\mu_2)$ is a $(W_2, \mathfrak p (1^+, 2^-))$-bounded weight. This implies that $(-\mu_1 - \mu_2 -1)\varepsilon_0 + (\mu_1 + 1)\varepsilon_1 + \mu_2\varepsilon_2$ is $(\mathfrak{sl}_3, s_{\varepsilon_0-\varepsilon_1} {\mathfrak b}_{st}^-)$-bounded. Here $s_{\beta}$ denotes the reflection of the Weyl group reflection corresponding to the $\mathfrak{sl}_3$-root $\beta$. We apply Lemma \ref{lem-sl3-bw} again but this time for the weight $(\mu_1 +1,\mu_2)$. Then one of the following conditions hold:\\
(a)  $ \mu_1 + 2\mu_2 + 1\in \Z_{\geq 0}$; (b) $ -2 \mu_1 - \mu_2 - 2 \in \Z_{\geq 0}$; (c) $\mu_2 - \mu_1 \in \Z_{\geq 0}$.\\
We already assumed that (c) does not hold. If (a) holds then $\mu$ is of Type 3. If (b) holds then $\mu$ can not be of Type 2. Hence, it remains to consider the case when $\mu$ is of Type 3. Look again at the simple highest weight $W_2$-module $L = L_{\mathfrak p (1^+, 2^-)} (\mu_1+1,\mu_2)$. As mentioned above, this module has a simple $\mathfrak{sl}_3$-subquotient $L_0$ with highest weight $(-\mu_1 - \mu_2 -1)\varepsilon_0 + (\mu_1 + 1)\varepsilon_1 + \mu_2\varepsilon_2$ relative to ${\mathfrak b}_{st}^-$. Since $\mu_2 - (-\mu_1 - \mu_2 -1) \in \Z_{\geq 0}$, $L_0$ is $\partial_2$-finite. Therefore $L$ has a simple $\mathfrak{sl}_3$-subquotient isomorphic to $L_{s_{\varepsilon_0-\varepsilon_1} {\mathfrak b}_{st}^-}^{\mathfrak{sl}}(\mu_2+1,-\mu_1-\mu_2-2)$. However one easily checks that since $\mu_2 - \mu_1 \notin \Z_{\geq 0}$ and $ \mu_1 + 2\mu_2 + 1\in \Z_{\geq 0}$, the weight $(\mu_2+1,-\mu_1-\mu_2-2)$ is not $(\mathfrak{sl}_3, s_{\varepsilon_0-\varepsilon_1}{\mathfrak b}_{st}^-)$-bounded. This contradicts with the fact that $L_{s_{\varepsilon_0-\varepsilon_1} {\mathfrak b}_{st}^-}^{\mathfrak{sl}}(\mu_2+1,-\mu_1-\mu_2-2)$ is a subquotient of the bounded module $L$.
\end{proof}

\begin{theorem} \label{th-hw-bounded} Let $\lambda \in \C^2$. Then the highest weight $W_2$-module $L_{\mathfrak b} (\lambda)$ is bounded if and only if:
\begin{itemize}
\item[(i)] $\lambda_1 - \lambda_2 \in \Z_{\geq 0}$ for $\mathfrak b = {\mathfrak p}(2^+,12^+)$ and  $\mathfrak b = {\mathfrak p}(1^-,12^-)$,
\item[(ii)] $\lambda_2 - \lambda_1 \in \Z_{\geq 0}$ for $\mathfrak b = {\mathfrak p}(1^+,12^+)$ and  $\mathfrak b = {\mathfrak p}(2^-,12^-)$,
\item[(iii)] $\lambda_1 - \lambda_2 + 1 \in \Z_{\geq 0}$ for $\mathfrak b = {\mathfrak p}(1^-,2^+)$.
\item[(iv)] $\lambda_2 - \lambda_1 + 1 \in \Z_{\geq 0}$ for $\mathfrak b = {\mathfrak p}(1^+,2^-)$,
\end{itemize}
\end{theorem}
\begin{proof} Using Lemma \ref{lem-hw-wbounded} and applying the duality functor $M\mapsto M^*$ and the twist by the automorphism $\sigma$, we easily prove (i) and (ii). 

Again by duality and because ${\mathfrak p}(1^+,2^-)^- = {\mathfrak p}(1^-,2^+)$, we see that it is enough to show (iii). For the ``only if'' direction we use that $L_{{\mathfrak p}(1^-,2^+)}(\lambda_1, \lambda_2)$ is isomorphic to a subquotient of the bounded module $T((\lambda_1+1,\lambda_2), (\lambda_1+1,\lambda_2), (1^-,2^+))$. Assume now that $L = L_{{\mathfrak p}(1^-,2^+)}(\lambda_1, \lambda_2)$ is bounded.  We reason as in the proof of  Lemma \ref{lem-sl3-bw}. Namely, we first observe that $L$ is $\partial_1$-injective. Then the module $D_{\langle \partial_1 \rangle} L / L$ has a ${\mathfrak p}(2^+, 12^+)$-primitive vector of weight $(\lambda_1 + 1, \lambda_2)$ (namely the vector $\partial_1^{-1}w$ where $w$ is a highest weight vector of $L$). Then we use (i) for  $(\lambda_1 + 1, \lambda_2)$ and ${\mathfrak p}(2^+, 12^+)$ and complete the proof.
\end{proof}

\subsection{Classification of simple bounded half-plane $W_2$-modules}
In this subsection we classify all simple bounded  $W_2$-modules $M$ whose supports are half-planes. Namely we give a necessary and sufficient conditions for the modules listed in  (i)--(vi) of Example \ref{ex-par-w2} and Proposition \ref{prop-w2-par} to be bounded. We call modules $M$ in that list ((i)--(vi)) \emph{simple weight half-plane modules}.

We first provide the decomposition of the half-plane tensor modules. It is not surprising that in this case the result is much more simple than the highest weight case described in Proposition \ref{prop-tensor-simple-w2}. The proof of the following proposition is provided in \cite{Cav}.

\begin{proposition}
Let $\nu,\lambda \in \C^2$ be such that ${\rm Int} (\lambda - \nu) = \{i\}$, where  $i=1$ or $i=2$. Let $J \in \mathcal{PM} (\lambda - \nu) $, i.e. $J \in \{ i^{+}, i^{-}\}$. Then $T(\nu,\lambda,J)$ is a simple $W_2$-module if and only $\lambda \neq (1,0)$.The module $T(\nu,(1,0),J)$ has length $2$ with simple submodule isomorphic to  $T(\nu,(0,0),J)$ and simple quotient isomorphic to $T(\nu,(1,1),J)$.
\end{proposition}

\begin{remark} \label{rem-char-half}
One easily can write character formulae for all tensor half-plane modules $T(\nu,\lambda,J)$. Naturally, these formulas contain more terms than the ones for highest weight modules, see Proposition \ref{prop-char-for}. For example, the character formula for $T(\nu,\lambda, 2^-)$ is:
$$
\ch T(\nu,\lambda, 2^-) = \frac{\left(\sum_{k \in \Z} e^{(\nu_1-\lambda_1 + k) \varepsilon_1 -\varepsilon_2}\right) \ch L_{\mathfrak{gl}} (\lambda)}{1- e^{-\varepsilon_2}}.
$$
\end{remark}

In this section we will use two parabolic induction functors. For a parabolic subalgebra $\mathfrak p = \mathfrak l  \oplus \mathfrak n^+$ of $W_2$ induced from a parabolic subalgebra of $\mathfrak{sl}_3$, and a simple $\mathfrak l$-module $S$ with trivial action of $\mathfrak n^+$, we define $M_{\mathfrak p} (S) = U(W_2) \otimes_{U(\mathfrak p)} S$. Also, by $L_{\mathfrak p} (S)$ we denote the simple quotient of $M_{\mathfrak p} (S)$. Similarly we define the two parabolic induction functors for the algebras $\mathcal A_1$, $\mathfrak{sl}_3$, and $\mathfrak{gl}_2$. We will use numerous times the facts that if $S$ is dense $\mathfrak{sl}_3$- or $\mathfrak{gl}_2$-module, then $S$ is a twisted localization of a bounded highest weight module, and that the twisted localization functors  commute with the parabolic induction functors $M_{\mathfrak p}$ and $L_{\mathfrak p}$, see for example Proposition 6.2 and Lemma 13.2 in \cite{M}. The proof that the twisted localization and the parabolic induction functors commute in \cite{M} concerns the case of a finite-dimensional reductive Lie algebra $\g$, but one naturally extends Mathieu's proof for $W_2$. For further properties and a more detailed exposition of the twisted localization functor, the reader is referred for example to \cite{Gr1}.

We first deal with  the last two cases in the list (i)--(vi) of Example \ref{ex-par-w2}.

\begin{lemma} Let $\mathfrak p = \mathfrak p (12^+)$ or $\mathfrak p = \mathfrak p (12^-)$,  and let $S$ be a simple $\mathfrak p$-module with a trivial action of $\mathfrak n^+$. Assume that the support of $M = L_{\mathfrak p} (S)$ is a half-plane. Then $M$ is not bounded.
\end{lemma}
\begin{proof} Assume that $M$ is bounded.   In both cases for $\mathfrak p$, the Levi subalgebra of $\mathfrak p$ is $\mathfrak a  \simeq \mathfrak{gl}_2$. Then since the support of $M$ is a half-plane, $S$ is a dense $x_2\partial_1$-injective $\mathfrak a$-module. So, let us consider $\lambda \in \h^*_{\mathfrak a} \simeq \C^2$ and $\nu \in \C^2$ so that $S = D_{\langle x_2\partial_1\rangle}^{\nu} L_{\mathfrak{gl}} (\lambda)$, where recall that $L_{\mathfrak{gl}} (\lambda)$ is the simple highest weight $\mathfrak a$-module relative to the Borel subalgebra  $\mathfrak b_{\mathfrak a}  = \Span \{x_1\partial_2, x_1\partial_1, x_2\partial_2\}$ of $\mathfrak a$. But then
$$
 L_{\mathfrak p} (S) \simeq  L_{\mathfrak p} \left( D_{\langle x_2\partial_1\rangle}^{\nu} L_{\mathfrak{gl}} (\lambda) \right) \simeq  D_{\langle x_2\partial_1\rangle}^{\nu} (L_{\mathfrak b}(\lambda)),
$$
where $\mathfrak b = \mathfrak b_{\mathfrak a} + \mathfrak{n}^+$. Since $L_{\mathfrak b}(\lambda)$ is bounded and  $\mathfrak b = {\mathfrak p}(2^+,12^+)$ or  $\mathfrak b = {\mathfrak p}(1^-,12^-)$, by Theorem \ref{th-hw-bounded}, we have that $\lambda_1 - \lambda_2 \in \mathbb Z_{\geq 0}$. But this implies that $L_{\mathfrak{gl}} (\lambda)$ is finite dimensional which contradicts to the fact that it  is $x_2 \partial_1$-injective.
\end{proof}

For the four remaining cases (i)--(iv) of simple bounded half-plane modules $L_{\mathfrak p} (S)$, the parabolic subalgebra $\mathfrak p$ has Levi component isomorphic to $\mathcal A$. More precisely, we have the following straightforward result.
\begin{lemma} \label{lem-levi-a}
The Levi component of $\mathfrak p = \mathfrak p (1^+)$ and $\mathfrak p = \mathfrak p (1^-)$ is $ \mathcal A (x_2)$, while the Levi component of $\mathfrak p = \mathfrak p (2^+)$ and $\mathfrak p = \mathfrak p (2^-)$ is $ \mathcal A (x_1)$.
\end{lemma}

Before we state our classification result for the bounded simple half-plane modules, recall that, by Theorem \ref{th-bounded-a}, every simple dense bounded weight $\A$-module is isomorphic to $T(\nu, \lambda, c)$ for some $\lambda,\nu,c$, such that $\lambda - \nu \notin \Z$. 

\begin{proposition} \label{prop-w2-bnd-nec} Let $\nu, \lambda,c \in \C$ be such that $\lambda - \nu \notin \mathbb Z$. Then the following isomorphisms hold. 
\begin{itemize}
\item[(i)] If $\lambda - c \in \Z_{\geq 0}$ and $(\lambda, c) \neq (1,0)$, then $L_{\mathfrak p (1^+)} T(\nu,\lambda,c) \simeq T((\lambda,\nu),(\lambda,c),1^+)$ and $L_{\mathfrak p (2^+)} T(\nu,\lambda,c) \simeq  T((\nu,\lambda),(\lambda,c),2^+)$.
Moreover, for any $\nu \notin \Z$, \\$L_{\mathfrak p (1^+)} T(\nu,1,0) \simeq  T((0,\nu),(0,0),1^+)$ and $L_{\mathfrak p (2^+)} T(\nu,1,0) \simeq  T((\nu,0),(0,0),2^+)$.

\item[(ii)]  If $c+ 1 - \lambda \in \Z_{\geq 0}$ and $(\lambda,c) \neq (0,0)$, then $L_{\mathfrak p (1^-)} T(\nu,\lambda,c) \simeq  T((\lambda,\nu),(c+1,\lambda),1^-)$ and $L_{\mathfrak p (2^-)} T(\nu,\lambda,c) \simeq  T((\nu,\lambda),(c+1,\lambda),2^-)$. Moreover, for any $\nu \notin \Z$,  $L_{\mathfrak p (1^-)} T(\nu,0,0) \simeq  T((1,\nu),(1,1),1^-)$ and $L_{\mathfrak p (2^-)} T(\nu,0,0) \simeq  T((\nu,1),(1,1),2^-)$.
\end{itemize}
\end{proposition}
\begin{proof}
We prove (i) for the parabolic subalgebra ${\mathfrak p} (1^+)$. The statements for the remaining three parabolic subalgebras are analogous. Let $\lambda - c \in \Z_{\geq 0}$ and $(\lambda, c) \neq (1,0)$. To show that $L_{\mathfrak p (1^+)} T(\nu,\lambda,c) \simeq  T((\lambda,\nu),(\lambda,c),1^+)$, observe that the nilradical of $\mathfrak p = \mathfrak p (1^+)$ is $\mathfrak n^+ = \Span \{ x_2^k\partial_1 \; | \; k \geq 0\}$. We easily check that  if  $x^s \otimes v \in T((\lambda,\nu),(\lambda,c),1^+)$ is such that $x_2^k \partial_1 (x^s \otimes v) = 0$, then $s_1 =  c$ and the weight of $v$ must be $(c,\lambda)$. Therefore the $\mathfrak n^+$-invariants of $T((\lambda,\nu),(\lambda,c),1^+)$ form an $\A (x_2)$-module isomorphic to $T(\nu,\lambda,c)$, which proves the desired isomorphism for ${\mathfrak p} (1^+)$. The isomorphism for ${\mathfrak p} (2^+)$ follows with similar reasoning. It remains to consider the case $(\lambda,c) =(1,0)$. In this case we use that $T(\nu,1,0)  \simeq T(\nu,0,0)$ as $\A (x_2)$-modules and apply the isomorphism we just proved for $(\lambda,c) =(0,0)$ (possible because $(\lambda,c) \neq (1,0)$). Part (ii) follows in a similar way.
\end{proof}

\begin{theorem} \label{th-half-plane} Let $\nu, \lambda,c \in \C$ be such that $\lambda - \nu \notin \mathbb Z$. The simple weight half-plane module $M \simeq L_{\mathfrak p} T(\nu,\lambda,c)$ is bounded if and only if the following conditions hold. 
\begin{itemize}
\item[(i)] $\lambda - c \in \Z_{\geq 0}$ for $\mathfrak p = \mathfrak p (1^+)$ or $\mathfrak p = \mathfrak p (2^+)$ .

\item[(ii)] $c+ 1 - \lambda \in \Z_{\geq 0}$ for $\mathfrak p = \mathfrak p (1^-)$ or $\mathfrak p = \mathfrak p (2^-)$ . 

\end{itemize}
\end{theorem}
\begin{proof}
The ``if'' directions follow from Proposition \ref{prop-w2-bnd-nec}. For  the ``only if'' directions, we prove again the condition only for the parabolic subalgebra $\mathfrak p(1^+)$  and then use similar reasoning for the remaining three parabolic subalgebras. We need to show that if $L_{\mathfrak p (1^+)} T(\nu,\lambda,c) $ is bounded, then $\lambda - c \in \Z_{\geq 0}$. If $(\lambda,c) = (1,0)$ the statement follows from the third isomorphism  of Proposition \ref{prop-w2-bnd-nec}(i). Assume now that $(\lambda,c) \neq (1,0)$. To prove the desired condition, we use that $T(\nu,\lambda,c) \simeq D_{\langle \partial_2 \rangle}^{\lambda - \nu} T(\lambda,\lambda,c,-)$,  see Lemma \ref{lem-iso-loc-a}(ii). Then
$$
L_{\mathfrak p (1^+)} T(\nu,\lambda,c) \simeq L_{\mathfrak p (1^+)} D_{\langle \partial_2 \rangle}^{\lambda - \nu} T(\lambda,\lambda,c,-) \simeq D_{\langle \partial_2 \rangle}^{\lambda - \nu}  L_{\mathfrak p (1^+)} L(\lambda - 1, c,-) \simeq D_{\langle \partial_2 \rangle}^{\lambda - \nu} L_{\mathfrak p (1^+,2^-)} (c, \lambda - 1).
$$
The last isomorphism uses the fact that the Levi subalgebra of $ L_{\mathfrak p (1^+)}$ is $\mathcal A (x_2)$, see Lemma \ref{lem-levi-a}. Hence $L_{\mathfrak p (1^+,2^-)} (c, \lambda - 1)$ is bounded and the condition $\lambda - c \in \Z_{\geq 0}$  follows from Theorem \ref{th-hw-bounded}(iv). \end{proof}

\subsection{Classification of  simple bounded dense  $W_2$-modules}

Recall that $M$ is a dense module if $\Supp M = \lambda + \Z^2$ for some $\lambda$. 
\begin{lemma} \label{lem-12-inj}
Let $M$ be a simple bounded $W_2$-module on which  $x_1 \partial_ 2$ or $x_2 \partial_1$ act finitely. Then the support of $M$ is contained in a horizontal or vertical half-plane. In particular, if $M$ is dense, then  $x_1 \partial_ 2$ and $x_2 \partial_1$ act injectively (hence bijectively) on $M$.
\end{lemma}
\begin{proof}
Assume that $M$ is not isomorphic to $\C$ and that $x_2 \partial_1$ acts finitely on $M$. To identify the possible types of  $M$ we use  representation theory of  $\mathfrak{gl}_2$. Let $\alpha = \varepsilon_1 - \varepsilon_2$. Recall that  ${\mathfrak a}\simeq \mathfrak{gl}_2$ is the subalgebra of $W_2$ generated by $x_i \partial_i$, $i,j = 1,2$. For a weight $\mu$ in $\Supp M$, consider the $\mathfrak a$-module $M[\mu] = \bigoplus_{k \in \Z} M^{\mu + k \alpha}$. This is a bounded $\mathfrak a$-module on which $x_2 \partial_1$ acts finitely. Then the number of weights of $(x_2\partial_1)$-primitive vectors in  $M[\mu]$ is bounded by $2 \deg M [\mu]$. Hence $M [\mu]$ has  an $(x_2\partial_1)$-maximal weight, say $\mu' = \mu + k \alpha$. Namely $\Supp M [\mu]$ is a subset  of the $\alpha$-half-line $\mu'+ {\mathbb Z}_{\leq 0}\alpha$. Thus the set $\left( \mu+ {\mathbb Z}_{\leq 0}\alpha\right) \cap \Supp M$ is also on the  $\alpha$-half-line $\mu'+ {\mathbb Z}_{\leq 0}\alpha$.  By Proposition \ref{prop-w2-par},  the possible supports of $M$ with empty $\alpha$-half-lines, are contained either in a horizontal, or a vertical, or a diagonal (i.e. $M$ is of type $12^+$  or $12^-$) half-plane. Assume now that  $M$ is of type $12^+$ (the case of  type $12^-$ is analogous). Then $M$ is a quotient of $U(W_2)\otimes_{U(\mathfrak p (12^+))} S$ for some simple $\mathfrak p (12^+)$-module $S$ whose support is a whole $\alpha$-line.  But, on the other hand,  $M$, and therefore $S$, is $x_2\partial_1$-finite. Thus $S$ can not be simple, which is a contradiction. \end{proof}

\begin{lemma} \label{lem-3-cases} Let $M$ be  a simple bounded dense $W_2$-module.

Then there is $\nu \notin {\mathbb Z}$ such that $M \simeq D_{\langle x_1\partial_2 \rangle}^{\nu} M_0$, where
\begin{itemize}
\item[(i)]  $M_0 = T(\nu',\lambda', 2^-)$ for some $\nu', \lambda'$ with $\lambda'_1- \nu'_1 \notin \Z$,  $\lambda'_2 - \nu'_2 \in \Z$, or
\item[(ii)]  $M_0 = T(\nu',\lambda', 1^+)$ for some $\nu', \lambda'$   with  $\lambda'_1- \nu'_1 \in \Z$,  $\lambda'_2 - \nu'_2 \notin \Z$, or
\item[(iii)]  $M_0 = T(\nu',\lambda', (1^+, 2^-))$  for some $\nu', \lambda'$  with $\lambda'_1- \nu'_1  \in \Z$,  $\lambda'_2 - \nu'_2  \in \Z$.
\end{itemize}
\end{lemma}
\begin{proof}
By Lemma \ref{lem-12-inj},   $x_1\partial_2$ and $x_2\partial_1$ act injectively on $M$. 
Let $\lambda = (\lambda_1, \lambda_2)$ be in $\Supp M$ and consider the module $D_{\langle x_1\partial_2 \rangle}^{-\nu} M$ for any $\nu \in \C$. For any $m \in M^{\lambda}$ we have
\begin{equation} \label{eq-comp-loc}
x_2\partial_1 \left( (x_1 \partial_2)^{-\nu} m\right) = (x_1 \partial_2)^{-\nu} \left( x_2\partial_1 + \nu(\lambda_1 - \lambda_2 - \nu -1) (x_1\partial_2)^{-1}\right)m.
\end{equation}
Consider the endomorphism $(x_1\partial_2)(x_2\partial_1)|_{M^{\lambda}}$ and choose an eigenvector $m$ with eigenvalue $x$. If we choose now $\nu$ to be a root of  $x + \nu(\lambda_1 - \lambda_2 - \nu -1) = 0$, we have that $x_2\partial_1 \left( (x_1 \partial_2)^{-\nu} m\right) = 0$. In particular, the submodule $\left( D_{\langle x_1\partial_2 \rangle}^{-\nu} M \right)^{\langle x_2 \partial_1\rangle}$ consisting of all $x_2\partial_1$-finite vectors in $D_{\langle x_1\partial_2 \rangle}^{-\nu} M$ is nonzero. Since this is a bounded module, by Proposition \ref{prop-fin-len} it has a simple submodule $M_0$. Then by Lemma \ref{lem-tw-loc-simple}, $M \simeq D_{\langle x_1\partial_2 \rangle}^{\nu} M_0$. Since, $M_0$ is bounded, $(x_1 \partial_2)$-injective, and $(x_2\partial_1)$-finite, then by  Lemma \ref{lem-12-inj}, $\Supp M_0$ is contained in a horizontal or vertical half-plane. But we classified all such modules in the last two subsections. After applying Theorem \ref{th-hw-bounded}, Proposition \ref{prop-w2-bnd-nec}, and Theorem \ref{th-half-plane}, we show that $M_0$ is indeed (exactly) one of the three types  listed in (i)--(iii).
\end{proof}

To achieve our goal, it remains to show that the modules $ D_{\langle x_1\partial_2 \rangle}^{\nu} M_0$ for all $M_0$ listed in (i)--(iii) of Lemma \ref{lem-3-cases} are tensor $W_2$-modules. The strategy is to identify each $M_0$ as a submodule of a twisted localization $ D_{\langle x_1\partial_2 \rangle}^{-\nu} T(s,\lambda)$ of a dense tensor module $T(s,\lambda)$. For this we will use that $M_0$ can easily be detected by the weights of its  $(x_2\partial_1)$-primitive  vectors. More precisely, for an associative algebra $\mathcal U$, $u \in \mathcal U$, and a weight $\mathcal U$-module $M$ (recall the definition of a weight $\mathcal U$-module in  \S\ref{subsec-wht}), let 
$$
\mbox{WP}_{M} (u) = \{\lambda \; | \; \exists m \in M^{\lambda}: um=0\}
$$ 
be the set of weights of all $u$-primitive weight vectors in $M$. If $M$ is fixed we will write $\mbox{WP} (u)$ for $\mbox{WP}_M (u)$. The following lemma concerns the case of  $\mathcal U = U(\mathfrak{gl}_2)$ and it follows from the representation theory of $\mathfrak{gl}_2$.
\begin{lemma} \label{lem-gl-2}
Let $\g = \mathfrak{gl}_2$, $\alpha = \varepsilon_1 - \varepsilon_2$,  $e \in \g^{\alpha}$, and $f \in \g^{-\alpha}$. If  $M$ is a weight $\mathfrak{gl}_2$-module for which $\ch M = \frac{e^{\lambda}}{1 - e^{\alpha}}$, then $\mbox{WP}_{D_{\langle e\rangle} M} (f) =  \{ \lambda, s_{\alpha}\cdot \lambda\}$ if $\lambda_1 - \lambda_2 \in \Z$ and $\mbox{WP}_{D_{\langle e\rangle} M} (f) =  \{ \lambda \}$, otherwise. Here $s_{\alpha} \cdot (\lambda_1,\lambda_2) = (\lambda_2-1,\lambda_1+1)$.
\end{lemma}

For convenience, in what follows we fix $\alpha = \varepsilon_1 - \varepsilon_2$ as a root of $\mathfrak a = \Span \{x_i\partial_j\; | \; i,j =1,2 \}$. To describe sets of  weights of primitive vectors of localized tensor modules we introduce some subsets of $\C^2$. For $y,z_1,z_2 \in \C$ with $z_1 - z_2 \in \Z_{\geq 0}$, set: \\$\mbox{Hor}(y,[z_1,z_2]) = \left(y + \Z\right) \times \left( [z_1,z_2] \cap (z_1 + \Z) \right)$ (horizontal strip in $(y,z_1) + \Z^2$), and  $\mbox{Ver}([z_1,z_2], y) =  \left( [z_1,z_2]  \cap (z_1 + \Z) \right) \times \left(y + \Z\right ) $ (vertical strip in $(z_1,y) + \Z^2$).


\begin{lemma} \label{lem-prim-wht} Let $M = D_{\langle x_1\partial_2 \rangle}M_0$ and $M_0$ be one of the three modules listed in (i)--(iii) in Lemma \ref{lem-3-cases}. Then the following hold. 
\begin{itemize}
\item[(i)] If $M_0 = T(\nu, \lambda, 2^-)$, then $\mbox{WP}_{M} (x_2\partial_1) = {\rm Hor} (\nu_1, [\lambda_1-1, \lambda_2-1])$.
\item[(ii)] If $M_0 = T(\nu, \lambda, 1^+)$, then $\mbox{WP}_{M} (x_2\partial_1) = {\rm Ver} ([\lambda_1, \lambda_2], \nu_2)$.
\item[(iii)] If $M_0 = T(\nu, \lambda, (1^+, 2^-))$, then $\mbox{WP}_{M} (x_2\partial_1) \subset {\rm Hor} (\nu_1, [\lambda_1-1, \lambda_2-1]) \cup {\rm Ver} ([\lambda_1, \lambda_2], \nu_2)$.
\end{itemize}
\end{lemma}
\begin{proof}
For part (i) we look at the character formula for $M_0$, see Remark \ref{rem-char-half}. More precisely, if $s \in \Supp M_0$, then the character of the $\mathfrak{a}$-module $M_0[s] = \oplus_{k \in \Z} M^{s+ k\alpha}$ equals the character of the $\mathfrak{gl_2}$-module $L_0 = \bigoplus_{i=0}^{\lambda_1-\lambda_2} L_{\mathfrak{gl}}(s_1+s_2 - \lambda_1- i+1, \lambda_1+i-1)$. But since the central characters of the  direct summands of $L_0$ are distinct, we have $M_0 [s] \simeq L_0$. Hence $\mbox{WP}_{M_0[s]} (x_2\partial_1)  = \mbox{WP}_{L_0} (E_{21})$.  We use  the identity  $M_0 = \bigoplus_{\ell \in \Z} M_0[s+\ell\varepsilon_1]$ of non-integral $\mathfrak{a}$-modules. After applying the functor $D_{\langle x_1\partial_2\rangle}$ on that identity, and Lemma \ref{lem-gl-2} on each $D_{\langle x_1\partial_2\rangle}M_0[s]$,  we prove the desired identity in part (i). The proof of part (ii) is similar to the proof of (i). For part (iii) we use the same reasoning, and in particular apply Lemma \ref{lem-gl-2} for the integral case. Unfortunately in this case, some direct summands of $M_0[s]$ may have equal central characters, so we can not claim that $M_0[s]$ is a direct sum of simple highest modules. For that reason, we cannot claim that identity holds for  $\mbox{WP}_{M} (x_2\partial_1)$, but we can prove that we have an inclusion. \end{proof}

In order to explicitly write  the weights of the $x_2\partial_1$-primitive vectors in $ D_{\langle x_1\partial_2 \rangle}^{-\nu} T(s,\lambda)$, we need to introduce additional  notation. Recall that the elementary matrices of $\mathfrak{gl}_2$ are denoted by $E_{ij}$, $i,j=1,2$. For the simple finite-dimensional $\mathfrak{gl}_2$-module $L_{\mathfrak{gl}} (\lambda)$ we  use the following setting: $L_{\mathfrak{gl}} (\lambda) = \Span \{ v_0,...,v_n \}$, $n=\lambda_1-\lambda_2$, with $\mathfrak{gl}_2$-action defined by
\begin{eqnarray*} 
E_{12} v_i & = & iv_{i-1},\\
E_{21} v_i & = & (n-i)v_{i+1},\\
E_{11} v_i & = & (\lambda_1 - i)v_{i},\\
E_{22} v_i & = & (\lambda_2+i)v_{i},
\end{eqnarray*}
for $i=0,...,n$.
For any $x,y \in \C$ we introduce the following $(n+1)\times(n+1)$ matrices:
\begin{equation}
A_n(x) = \begin{bmatrix}
    x  & 0 & \dots  & 0 & 0 \\
    1 & x-1& \dots  & 0 & 0 \\
    \vdots & \vdots  & & \vdots & \vdots \\
   0 & 0 & \dots  & n & x-n
\end{bmatrix}, \; 
B_n(y) = \begin{bmatrix}
    y-n & 0 & \dots  & 0 & 0 \\
    n & y-n+1& \dots  & 0 & 0 \\
    \vdots & \vdots  & & \vdots & \vdots \\
   0 & 0 & \dots  & 1 & y
\end{bmatrix}.
\end{equation}
In particular, $B_n(x)$ is the anti-diagonal transpose of $A_n(x)$.

\begin{lemma} \label{lem-ab-eigen}
Let $s, \lambda$ be in $\C^2$, and let $T = T(s,\lambda)$. 
\begin{itemize}
\item[(i)] The matrix $M(s,\lambda) $ of the endomorphism $(x_1\partial_2)(x_2\partial_1)|_{T^{s}}$ relative to the basis $\{ x^s\otimes v_0,...,x^s\otimes v_n \}$ of $T^{s} = T(s,\lambda)^s$ is
$$
M(s,\lambda) := A_n(s_2-\lambda_1+1)B_n(s_1-\lambda_2).
$$ 
\item[(ii)] Let $v \in L_{\mathfrak{gl}}(\lambda)$ be such that $x_1 \partial_2$ is injective on $x^s \otimes v$. Then for all $\nu \in \C$, 
$$
(x_2\partial_1) (x_1 \partial_2)^{-\nu} (x^s \otimes v)=  (x_1 \partial_2)^{-\nu-1} \left( (x_1\partial_2) (x_2 \partial_1) + \nu(s_1-s_2 - \nu-1) \Id\right) (x^s \otimes v)
$$
in $D_{\langle x_1\partial_2 \rangle }^{-\nu} T (s,\lambda)$.
\item[(iii)]  The characteristic polynomial of $A_n(x)B_n(y)$ is:
\begin{equation} \label{eq-char-poly}
\det \left( \mu I_{n+1}- A_n(x)B_n(y) \right) = (\mu - xy)(\mu - (x-1)(y-1))...(\mu - (x-n)(y-n)).
\end{equation}
\end{itemize}
\end{lemma}
\begin{proof}
Part (i) follows by a direct verification using the formulas (\ref{def-tensor-action}) and the explicit $\mathfrak{gl}_2$-action on $L_{\mathfrak{gl}}(\lambda)$. Part (ii) is also a subject of direct verification (see also (\ref{eq-comp-loc})). 

We can prove part (iii) with purely technical tools, but there is a more elegant proof using the structure of the modules $D_{\langle x_1\partial_2 \rangle }^{-\nu}  T(s,\lambda)$.  Let $P(\mu,x,y)$ be the characteristic polynomial of $A_n(x)B_n(y)$. Note that  $P(\mu,x,y)$  is polynomial in $\mu,x,y$.

Consider  $s,\lambda$ such that  $s_1 - \lambda_1 \notin \Z$ and $s_2 - \lambda_2 \in \Z$, and let $\nu \in \Z$. In this case $x_1 \partial_2$ is injective on $x^s \otimes v$ if and only if  $x^s \otimes v \notin T(s,\lambda, 2^+)$. If the latter holds, by part (ii) we have that $s-\nu\alpha$ is a weight of an $(x_2\partial_1)$-primitive vector  in $D_{\langle x_1\partial_2 \rangle }^{-\nu}  T(s,\lambda)$ if and only if  $\nu(\nu+1-s_1+s_2)$ is an eigenvalue of $M(s,\lambda)$. On the other hand, since $L \mapsto D_{\langle x_1\partial_2 \rangle }L$ is an exact  functor, $\nu \in \Z$, and $D_{\langle x_1\partial_2 \rangle } T(s,\lambda, 2^+) =0$, we have
$$
D_{\langle x_1\partial_2 \rangle }^{-\nu}  T(s,\lambda) \simeq D_{\langle x_1\partial_2 \rangle } T(s,\lambda) \simeq D_{\langle x_1\partial_2 \rangle } T(s,\lambda, 2^-).
$$ 
By Lemma \ref{lem-prim-wht}(i), 
$$
(s_1+s_2 - \lambda_2-i+1, \lambda_2+i-1) = s - (\lambda_2-s_2+i-1)\alpha, \; i=0,...,\lambda_1-\lambda_2,
$$ are the weights of the set of $(x_2\partial_1)$-primitive vectors of $D_{\langle x_1\partial_2 \rangle } T(s,\lambda, 2^-)$ that are on the $\alpha$-line $s+ \Z\alpha$. Therefore, $xy$, $(x-1)(y-1)$,... $(x-n)(y-n)$ are all eigenvalues (with possible repetitions) of $(x_1\partial_2) (x_2 \partial_1)|_{T^s}$, where $x = s_2 - \lambda_2 + 1$ and $y = s_1 - \lambda_2$. Thus (\ref{eq-char-poly}) holds for all $\mu$, $x \in \Z$, $y \notin \Z$. Since $\Z \times \left( \C\setminus \Z\right)$ is a Zariski dense subset of $\C^2$, we have that (\ref{eq-char-poly})  holds for all $\mu,x,y$. \end{proof}

\begin{lemma} \label{lem-reverse-loc}
Let $\lambda, s \in \C^2$ be such that $\lambda_i - s_i \notin \Z$, $i=1,2$, and $\lambda \neq (1,0)$.
\begin{itemize}
\item[(i)] If $\lambda_1 + \lambda_2 - s_1 - s_2 \notin \Z$, then the following isomorphisms hold:
\begin{itemize}
\item[(a)] $D_{\langle x_1\partial_2\rangle}^{\nu_2}T(s- \nu_2\alpha, \lambda, 2^-) \simeq T(s, \lambda)$,
\item[(b)] $D_{\langle x_1\partial_2\rangle}^{\nu_1}T(s- \nu_1\alpha, \lambda, 1^+) \simeq T(s, \lambda)$,
\end{itemize} 
for  $\nu_2 = s_2-\lambda_2 + 1$ and $\nu_1 = s_1- \lambda_1$.
\item[(ii)] If $\lambda_1 + \lambda_2 - s_1 - s_2 \in \Z$, then 
$$
D_{\langle x_1\partial_2\rangle}^{\nu}T(s- \nu \alpha, \lambda, (1^+,2^-)) \simeq T(s, \lambda)
$$
for  $\nu_2 = s_2-\lambda_2 + 1$ and for $\nu_1 = s_1-\lambda_1$.
\end{itemize}
\begin{proof}
For part (i)(a) consider first the module $D_{\langle x_2\partial_2\rangle}^{-\nu_2}T(s, \lambda)$ (where $\nu_2 = s_2-\lambda_2$), and let  $\nu \in \C$ be such that  $\nu - \nu_2 \in \Z$.  By Lemma \ref{lem-ab-eigen}(ii) we know that $s-\nu\alpha$ is a weight of an $(x_2\partial_1)$-primitive vector  in $D_{\langle x_1\partial_2 \rangle }^{-\nu_2}  T(s,\lambda)$ if and only if  $\nu(\nu+1-s_1+s_2)$ is an eigenvalue of $M(s,\lambda)$. But by Lemma \ref{lem-ab-eigen}(ii), all such eigenvalues are $(s_2-\lambda_2-i+1)(s_1-\lambda_2-i)$, $i=0,...,\lambda_1-\lambda_2$. Recall that $\nu - s_2+\lambda_2 \in \Z$ and $\lambda_1 + \lambda_2 - s_1 - s_2 \notin \Z$. Hence,  $s-\nu\alpha$ is a weight of an $(x_2\partial_1)$-primitive vector  in $D_{\langle x_1\partial_2 \rangle }^{-\nu_2}  T(s,\lambda)$ if and only if $\nu = s_2-\lambda_2-i+1$ for some $i \in \{0,1,...,\lambda_1-\lambda_2\}$.

On the other hand,  $T(s,\lambda)$ is dense and $D_{\langle x_1\partial_2 \rangle }^{-\nu_2}  T(s,\lambda)$ has $(x_2\partial_1)$-primitive vectors. Thus  by Lemma \ref{lem-3-cases}, we have $T(s,\lambda) = D_{\langle x_1\partial_2 \rangle }^{\nu_2}  T(s-\nu_2\alpha,\lambda', 2^-)$ for some $\lambda' \in \C^2$. Hence, $D_{\langle x_1\partial_2 \rangle }^{-\nu_2}  T(s,\lambda) \simeq  D_{\langle x_1\partial_2 \rangle } T(s-\nu_2\alpha,\lambda', 2^-)$. Using Lemma \ref{lem-prim-wht}(i) and the description of the weights of $(x_2\partial_1)$-primitive vectors  in $D_{\langle x_1\partial_2 \rangle }^{-\nu_2}  T(s,\lambda)$,  we obtain 
$$
 {\rm Hor} (s_1 - \nu_2, [\lambda_1-1, \lambda_2-1]) =  {\rm Hor} (s_1 - \nu_2, [\lambda'_1-1, \lambda'_2-1]).
 $$
 Thus $\lambda = \lambda'$, which completes the proof of (i)(a). 

For parts (i)(b) and (ii) we use the same reasoning, namely we apply again the corresponding parts of  Lemma \ref{lem-ab-eigen}, Lemma \ref{lem-3-cases},  Lemma \ref{lem-prim-wht}. Part (i)(b) will automatically follow, while for part (ii) we will have at the end
$$
 {\rm Hor} (s_1 - \nu_2, [\lambda_1-1, \lambda_2-1]) \cup {\rm Ver} ([\lambda_1, \lambda_2], s_2 + \nu_2) \subset  {\rm Hor} (s_1 - \nu_2, [\lambda'_1-1, \lambda'_2-1]) \cup {\rm Ver} ([\lambda'_1, \lambda'_2], s_2 + \nu_2).
$$
However, the above condition is sufficient to conclude that $\lambda = \lambda'$. \end{proof}
\end{lemma}

Using Lemma \ref{lem-3-cases} and Lemma \ref{lem-reverse-loc}, we obtain the classification of simple bounded dense $W_2$-modules.
\begin{theorem} \label{th-dense}
If $M$ is a simple bounded dense $W_2$-module, then $M \simeq T(\nu,\lambda)$ for some $\nu, \lambda$, such that $\lambda_i - \nu_i \notin \Z$, $i=1,2$, $\lambda \neq (1,0)$. 
\end{theorem}

\subsection{Main Theorem}
Combining Theorems \ref{th-hw-bounded}, \ref{th-half-plane}, and \ref{th-dense} we obtain our main result in the paper.

\begin{theorem} \label{th-main}
Let $M $ be a simple bounded $W_2$-module. Then either $M \simeq \C$ or  $M \simeq T(\nu,\lambda, J)$ for some $\nu,\lambda \in \C^2$, and $J \in \mathcal{PM} (\lambda - \nu)$, such that: 
$$
\lambda \neq (1,0), (\nu,\lambda, J) \neq ((0,0),(0,0),(1^+,2^+)), (\nu,\lambda, J) \neq ((1,1),(1,1),(1^-,2^-)).
$$
 Furthermore, two modules $T(\nu,\lambda, J) $ and $T(\nu',\lambda', J')$ in this list are isomorphic if and only if $\nu - \nu' \in \Z^2$, $\lambda = \lambda'$, and $J = J'$.
\end{theorem}

\end{document}